\documentclass[11pt,english]{amsart}
\usepackage[T1]{fontenc}
\usepackage[latin9]{inputenc}
\usepackage{geometry}
\usepackage{float}
\usepackage{mathrsfs}
\usepackage{amstext}
\usepackage{amsthm}
\usepackage{amssymb}
\usepackage{graphicx}
\usepackage{setspace}
\usepackage{wasysym}
\makeatletter
\numberwithin{equation}{section}
\numberwithin{figure}{section}
\theoremstyle{plain}
\newtheorem{thm}{\protect\theoremname}[section]
  \theoremstyle{definition}
  \newtheorem{defn}[thm]{\protect\definitionname}
  \theoremstyle{plain}
  \newtheorem{cor}[thm]{\protect\corollaryname}
  \theoremstyle{plain}
  \newtheorem{conjecture}[thm]{\protect\conjecturename}
  \theoremstyle{plain}
  \newtheorem{prop}[thm]{\protect\propositionname}
  \theoremstyle{remark}
  \newtheorem{rem}[thm]{\protect\remarkname}
  \theoremstyle{plain}
  \newtheorem{fact}[thm]{\protect\factname}
  \theoremstyle{plain}
  \newtheorem{lem}[thm]{\protect\lemmaname}
  \theoremstyle{plain}
  \newtheorem*{thm*}{\protect\theoremname}

\date{}
\usepackage{mathrsfs}
\pdfpageattr{/Group << /S /Transparency /I true /CS /DeviceRGB>>}
\usepackage{ae,aecompl} 
\usepackage{enumerate}
\usepackage{bbm}
\usepackage{hyperref}

\let\originalleft\left
\let\originalright\right
\renewcommand{\left}{\mathopen{}\mathclose\bgroup\originalleft}
\renewcommand{\right}{\aftergroup\egroup\originalright}

\hypersetup{
                    colorlinks=true,
                    linkcolor=red,
                    citecolor=blue,
                    urlcolor=green,
                    linktocpage=true,
                    pdfstartview=FitH,
                }

\makeatother

\usepackage{babel}
  \providecommand{\conjecturename}{Conjecture}
  \providecommand{\corollaryname}{Corollary}
  \providecommand{\definitionname}{Definition}
  \providecommand{\factname}{Fact}
  \providecommand{\lemmaname}{Lemma}
  \providecommand{\propositionname}{Proposition}
  \providecommand{\remarkname}{Remark}
  \providecommand{\theoremname}{Theorem}
\providecommand{\theoremname}{Theorem}

\begin{document}
\global\long\def\d#1{\,{\rm d}#1}

\global\long\def\R{\mathbb{R}}

\global\long\def\C{\mathbb{C}}

\global\long\def\Z{\mathbb{Z}}

\global\long\def\N{\mathbb{N}}

\global\long\def\Q{\mathbb{Q}}

\global\long\def\T{\mathbb{T}}

\global\long\def\F{\mathbb{F}}

\global\long\def\vp{\varphi}

\global\long\def\Sph{\mathbb{S}}

\global\long\def\sub{\subseteq}

\global\long\def\one{\mathbbm1}

\global\long\def\vol#1{\text{vol}\left(#1\right)}

\global\long\def\EE{\mathbb{E}}

\global\long\def\sp{{\rm sp}}

\global\long\def\iprod#1#2{\langle#1,\,#2\rangle}

\global\long\def\uball{B_{2}^{n}}

\global\long\def\conv#1{{\rm conv}\left(#1\right)}

\global\long\def\met#1{{\rm M}\left(#1\right)}

\global\long\def\supp#1{{\rm supp}\left(#1\right)}

\global\long\def\eps{\varepsilon}

\global\long\def\PP{\mathbb{P}}

\global\long\def\vein#1{{\rm vein}\left(#1\right)}

\global\long\def\fvein#1{{\rm vein^{*}}\left(#1\right)}

\global\long\def\bfvein#1{{\rm \mathbf{vein^{*}}}\left(\mathbf{#1}\right)}

\global\long\def\ill#1{{\rm ill\left(#1\right)}}

\global\long\def\fill#1{{\rm ill^{*}\left(#1\right)}}

\global\long\def\ovr#1{{\rm ovr\left(#1\right)}}

\global\long\def\norm#1{\left\Vert #1\right\Vert }

\global\long\def\proj{{\rm Pr}}

\title{Approximations of convex bodies by measure-generated sets}

\author{Han Huang }

\address{Department of Mathematics, University of Michigan, Ann Arbor, MI.}

\email{sthhan@umich.edu (H. Huang).}

\author{Boaz A. Slomka}

\address{Department of Mathematics, University of Michigan, Ann Arbor, MI.}

\email{bslomka@umich.edu (B. A. Slomka).}
\begin{abstract}
Given a Borel measure $\mu$ on $\R^{n}$, we define a convex set
by
\[
\met{\mu}=\bigcup_{\substack{0\le f\le1,\\
\int_{\R^{n}}f\d{\mu}=1
}
}\left\{ \int_{\R^{n}}yf\left(y\right)\d{\mu}\left(y\right)\right\} ,
\]
where the union is taken over all $\mu$-measurable functions $f:\R^{n}\to\left[0,1\right]$
with $\int_{\R^{n}}f\d{\mu}=1$. We study the properties of these
measure-generated sets, and use them to investigate natural variations
of problems of approximation of general convex bodies by polytopes
with as few vertices as possible. In particular, we study an extension
of the vertex index which was introduced by Bezdek and Litvak. As
an application, we provide a lower bound for certain average norms
of centroid bodies of non-degenerate probability measures. 
\end{abstract}

\maketitle

\section{Introduction}

\subsection{Background and motivation }

Problems pertaining to approximation, on their various aspects and
applications, have been extensively studied in the theory of convex
bodies, see e.g. \cite{Gr93}, and \cite{Br07}. 

An example for such a problem is that of approximating a convex body,
namely a compact convex set with non-empty interior, by a polytope
(the convex hull of finitely many points) with as few vertices as
possible, within a given Banach-Mazur distance. More precisely, for
any convex body $K\sub\R^{n}$ centered at the origin, and $R>1$
we define:

\[
d_{R}\left(K\right)=\inf\left\{ N\in\N\,:\,\exists\,P=\conv{x_{1},\dots,x_{N}}\sub\R^{n}\,,\,\frac{1}{R}P\sub K\sub P\right\} ,
\]
where $\conv{x_{1},\dots,x_{N}}$ is the convex hull of $x_{1},\dots,x_{N}$.
\\

A result of Barvinok \cite{Barvinok14} implicitly states that for
any centrally-symmetric convex body $K\sub\R^{n}$ and $2<R<\sqrt{n}$,
\begin{equation}
d_{R}\left(K\right)\le e^{cn\log R/R^{2}}\label{eq:Barvinok_d}
\end{equation}
for some universal constant $c>0$. In particular, $d_{c\sqrt{n}}\left(K\right)\le n$.
We also mention the result of Szarek \cite{Szarek2014} who shows
that for any convex body with center of mass at the origin and $2<R<n$,
\begin{equation}
d_{R}\left(K\right)\le ne^{cn/R}.\label{eq:Szarek_d}
\end{equation}

\noindent For the case $R=n$, a similar result to that of Szarek
can be found in \cite{Bra+17}. 

We remark that both approaches in \cite{Barvinok14} and \cite{Ba+08}
work in the fine scale regime, for which an optimal result was very
recently proven in \cite{NNR17}. 

It is also worth pointing out that there is still a large gap between
the symmetric and the non-symmetric case. For example, it is not clear
whether in the non-symmetric case $d_{\sqrt{n}}\left(K\right)$ can
have a polynomial bound in $n$.\\

Note that for the special case $R=\infty$, $d_{\infty}\left(K\right)$
trivially equals $n+1$, e.g., by scaling away the vertices of a centered
simplex. However, replacing the number of vertices of the approximating
polytope by a different ``cost'' leads to the following quantity:
\[
D_{R}\left(K\right)=\inf\left\{ \sum_{i=1}^{N}\norm{x_{i}}_{K}\,:\,\exists\,P=\conv{x_{1},\dots,x_{N}}\sub\R^{n}\,,\,\frac{1}{R}P\sub K\sub P\right\} .
\]
Here, $\norm{\cdot}_{K}$ stands for the gauge function of $K$ which,
in the case where $K=-K$, is the norm on $\R^{n}$ which is induced
by $K$. This quantity is also linear-invariant, and is equivalent
to $d_{R}\left(K\right)$ for any finite $R$ in the sense that $d_{R}\left(K\right)\le D_{R}\left(K\right)\le Rd_{R}\left(K\right)$.
However, $D_{\infty}\left(K\right)$ is no longer trivial. In fact,
it coincides with the vertex index of $K$, denoted by $\vein K$,
which was introduced by Bezdek and Litvak in \cite{BezLitvak07},
and further studied in \cite{GluLit08} and \cite{GlusLitvak12}.
For example, it was shown that for any centrally-symmetric convex
body $K\sub\R^{n}$, 
\[
2n\le\vein K\le24n^{3/2}
\]
The lower bound, which is attained for $B_{1}^{n}$, was proved in
\cite{GluLit08}, and the upper bound, which (up to a universal constant)
is attained for the Euclidean unit ball $B_{2}^{n}$, was proved in
\cite{GlusLitvak12}. We remark that the choice of the $l_{1}$ cost
$\sum_{i=1}^{N}\norm{x_{i}}_{K}$ seems arbitrary and can be replaced
by different linear-invariant costs, such as $\left(\sum\norm{x_{i}}_{K}^{p}\right)^{1/p}$
for any $p\ge1$. 

\subsection{Metronoids}

The main purpose of this note is to introduce a natural way of generating
convex bodies from Borel measures, along with associated costs, and
study new quantities which are closely related to $D_{R}\left(K\right)$,
$d_{R}\left(K\right)$, and $\vein K$. Our construction goes as follows: 
\begin{defn}
\label{def:convmu} Given a Borel measure $\mu$ on $\R^{n}$, we
define 

\[
\met{\mu}=\bigcup_{\substack{0\le f\le1,\\
\int_{\R^{n}}f\d{\mu}=1
}
}\left\{ \int_{\R^{n}}yf\left(y\right)\d{\mu}\left(y\right)\right\} ,
\]
where the union is taken over all measurable functions $f:\R^{n}\to\left[0,1\right]$
with $\int_{\R^{n}}f\d{\mu}=1$. We call the set $\met{\mu}\sub\R^{n}$,
the \textsl{metronoid}\footnote{originating from the greek word ``metron'' for ``measure'' (the
authors thank B. Vritsiou for the greek lesson).} generated by $\mu$
\end{defn}

Note that $\met{\mu}$ is always a closed convex set, which is bounded
if $\mu$ has finite first moment. In particular, for $x_{1},\dots,x_{N}\in\R^{n}$
the discrete measure $\mu=\sum_{i=}^{N}\delta_{x_{i}}$ generates
the convex hull of $\left\{ x_{1},\dots,x_{N}\right\} $. By adding
weights, $w_{1},\dots,w_{N}>0$, that is, considering the weighted
measure $\mu=\sum_{i=1}^{N}w_{i}\delta_{x_{i}}$, the generated convex
body $\met{\mu}$ becomes a ``weighted convex hull'', where each
point $x_{i}$ can only participate in the convex hull with a coefficient
$\lambda_{i}$ whose maximal value is $w_{i}$ . In other words, we
have:
\[
\met{\mu}=\left\{ x\in\R^{n}\,:\,\exists\left\{ \lambda_{i}\right\} _{i=1}^{N}\,\text{s.t.}\,0\le\lambda_{i}\le w_{i},\,\sum_{i=1}^{N}\lambda_{i}=1,\,\text{and }x=\sum_{i=1}^{N}\lambda_{i}x_{i}\right\} .
\]

\noindent Also note that if $\mu\left(\R^{n}\right)<1$ then $\met{\mu}=\emptyset$,
and if $\mu\left(\R^{n}\right)=1$ then $\met{\mu}$ is the singleton
$\left\{ \int x\d{\mu}\left(x\right)\right\} $, namely the center
of mass of $\mu$. 

\noindent One may consider various other interesting classes of metronoids,
for example, the class of bodies generated by uniform measures on
convex bodies, which turn out to be closely related to floating bodies.
For detailed discussion on special classes of metronoids and their
properties, see Section \ref{sec:Met_prop} below. 

The notion of metronoids leads to the following variations of $d_{R}\left(K\right)$
and $D_{R}\left(K\right)$: 

\[
d_{R}^{*}\left(K\right)=\inf\left\{ \mu\left(\R^{n}\right)\,:\,\frac{1}{R}\met{\mu}\sub K\sub\met{\mu}\right\} ,
\]
 and 
\[
D_{R}^{*}\left(K\right)=\inf\left\{ \int_{\R^{n}}\norm x_{K}\d{\mu}\left(x\right)\,:\,\frac{1}{R}\met{\mu}\sub K\sub\met{\mu}\right\} .
\]

\noindent Clearly, we have that $d_{R}^{*}\left(K\right)\le d_{R}\left(K\right)$,
and $D_{R}^{*}\left(K\right)\le D_{R}\left(K\right)$. One can also
verify that the above quantities are both linear-invariant. While
it is plausible that the family of metronoids generated by all finite
Borel measure coincides with the family of all convex bodies, it is
still interesting to consider the approximation by metronoids since
for different values of $R$, the associated costs $\mu\left(\R^{n}\right)$
and $\int_{\R^{n}}\norm x_{K}\d{\mu}$ are not necessarily minimized
for $\mu$ for which $\met{\mu}=K$. 

For $R=\infty$, we obtain the following variation of the vertex index,
which we refer to as the fractional vertex index:

\[
\fvein K=\inf\left\{ \int_{\R^{n}}\norm x_{K}d\mu\left(x\right)\,:\,K\sub\met{\mu}\right\} .
\]

\noindent We remark that the motivation of Bezdek and Litvak to study
the vertex index is its relation to Hadwiger's famous problem of illuminating
a convex body by light sources, and to the Gohberg-Markus-Hadwiger
equivalent problem of covering a convex body by smaller copies of
itself (see e.g., \cite{BrassMoser05} and references therein). Fractional
versions of the illumination and covering problems were studied in
\cite{Naszodi09} and \cite{ArtSlom14-FracCov}. 

\subsection{Main results}

\subsubsection{\textbf{Upper and lower bounds.} Our first main result provides a
bound for $d_{\sqrt{n}}^{*}\left(K\right)$ and $D_{\sqrt{n}}^{*}\left(K\right)$
in the centrally-symmetric case: }
\begin{thm}
\label{thm:main_upper_sym}There exists a universal constant $C>0$
such that for every centrally-symmetric convex body $K\sub\R^{n}$,
one has 
\[
d_{\sqrt{n}}^{*}\left(K\right)\le C,\,\,\text{and}\,\,\,D_{\sqrt{n}}^{*}\left(K\right)\le Cn.
\]
\end{thm}

\noindent Our second main result provides a general upper for $d_{R}^{*}\left(K\right)$
and $D_{R}^{*}\left(K\right)$:
\begin{thm}
\label{thm:main_upper_bdd}Let $K\sub\R^{n}$ be a centered convex
body. Then for $1<R\le n$ one has 
\[
d_{R}^{*}\left(K\right)\le\exp\left(1+\frac{n-1}{R-1}\right),\,\,\text{and}\,\,D_{R}^{*}\left(K\right)\le R\exp\left(1+\frac{n-1}{R-1}\right).
\]
\end{thm}

Note that Theorems \ref{thm:main_upper_sym} and \ref{thm:main_upper_bdd}
are reminiscent of \ref{eq:Barvinok_d} and \ref{eq:Szarek_d}, but
do not follow from them formally. We believe that a further investigation
of $d_{R}^{*}\left(K\right)$ and $D_{R}^{*}\left(K\right)$ in the
non-symmetric case may shed light on the classical counterpart $d_{R}\left(K\right)$
and $D_{R}\left(K\right)$, e.g., in the regime of $R\approx\sqrt{n}$.

An immediate corollary of Theorem \ref{thm:main_upper_bdd} provides
the following upper bound for the fractional vertex index,
\begin{cor}
\label{cor:upper_bdd}For every centered convex body $K\sub\R^{n}$
one has $\fvein K\le D_{n}^{*}\left(K\right)\le e^{2}n$. 
\end{cor}

Our third main result provides a lower bound for the fractional vertex
index in the centrally-symmetric case
\begin{thm}
\label{thm:fvein_lower_bdd}There exists a universal constant $c>0$
such that for every centrally-symmetric convex body $K\sub\R^{n}$,
one has:
\[
\fvein K\ge c\sqrt{n}.
\]
\end{thm}

We remark that, up to a constant, Corollary \ref{cor:upper_bdd} is
sharp for the cross-polytope $B_{1}^{n}=\left\{ \left(x_{1},\dots,x_{n}\right)^{T}\in\R^{n}\,:\,\sum_{i=1}^{n}\left|x_{i}\right|\le1\right\} $,
and Theorem \ref{thm:fvein_lower_bdd} is sharp for the Euclidean
unit ball $B_{2}^{n}=\left\{ x\in\R^{n}\,:\,\left|x\right|^{2}\le1\right\} $.
In fact, in Section \ref{sec:fvein_exact_comp} we show that $\fvein{B_{1}^{n}}=2n$
, and $\fvein{B_{2}^{n}}=\sqrt{2\pi n}\left(1+{\rm o}\left(1\right)\right)$.

The proof of Theorem \ref{thm:fvein_lower_bdd} employs a proportional
Dvoretzky-Rogers factorization Theorem by Bourgain and Szarek \cite{BourSz88}.
However, we suspect that a proof by symmetrization should show that
the extremizer in the general case is $B_{2}^{n}$:
\begin{conjecture}
For any centered convex body $K\sub\R^{n}$, $\fvein K\ge\fvein{B_{2}^{n}}\approx\sqrt{n}$. 
\end{conjecture}

\subsubsection{\textbf{An application to centroid bodies.} }

\noindent The $L_{p}$-centroid bodies were introduced by Lutwak and
Zhang \cite{LuZh97} (under different normalization than we use below)
and have been studied extensively by various authors. In particular,
$L_{p}$-centoid bodies have become an indispensable part of the theory
of asymptotic convex geometry since the seminal work of Paouris \cite{Paouris2006}.
For a survey on this subject, see \cite[Ch. 5]{BGVV14}, and references
therein.

Given $p\ge1$ and a Borel probability measure $\mu$ with bounded
$p^{th}$ moment, the $L_{p}$-centroid body $Z_{p}\left(\mu\right)$
is defined by the relation 
\[
h_{Z_{p}\left(\mu\right)}\left(\theta\right)=\left(\int_{\R^{n}}\left|\iprod x{\theta}\right|^{p}\d{\mu}\left(x\right)\right)^{1/p},
\]
where $\iprod{\cdot}{\cdot}$ stands for the standard Euclidean inner
product on $\R^{n}$, and $h_{K}\left(\theta\right)=\sup_{K}\iprod x{\theta}$
is the support function of a convex body $K\sub\R^{n}$ (see e.g.,
\cite{SchBook} for properties of supporting functionals). 

For a log-concave measure $\mu$, the bodies $Z_{p}\left(\mu\right)$
admit many remarkable properties due to the phenomenon of concentration
of measure. For example, reverse Hölder inequalities for norms, which
imply that, for some universal constant $c>0$, $Z_{p}\left(\mu\right)\sub Z_{q}\left(\mu\right)\sub c\frac{q}{p}\,Z_{p}\left(\mu\right)$
for any $1\le p\le q$. Moreover, for $p\ge1$, one has 
\[
\left(\int_{\R^{n}}\norm x_{Z_{p}\left(\mu\right)}^{p}\d{\mu}\left(x\right)\right)^{1/p}\apprge\frac{\sqrt{n}}{p}.
\]

It turns out that for $p=1$, the above estimation holds without the
assumption that $\mu$ is log-concave. In fact, this result is a direct
corollary of Theorem \ref{thm:fvein_lower_bdd}: 
\begin{cor}
\label{cor:centrod_low_bdd}There exists a universal constant $c>0$
such that for any non-degenerate probability Borel measure $\mu$
with bounded first moment, one has 
\[
\int_{\R^{n}}\norm x_{Z_{1}\left(\mu\right)}\d{\mu}\left(x\right)\ge c\sqrt{n}.
\]
\end{cor}

We remark that the proof of Corollary \ref{cor:centrod_low_bdd} (or,
equivalently, of Theorem \ref{thm:fvein_lower_bdd}) is based on high-dimensional
phenomena, rather than concentration of measure (which is used to
obtain the same result in the case of log-concave measures). Other
results in the spirit of Corollary \ref{cor:centrod_low_bdd}, where
the log-concavity assumption on the measure may be relaxed, can be
found in \cite{Klartag10,P12,K12,K15,KZ15}. 

This paper is organized as follows. In Section \ref{sec:Met_prop}
we study the properties of metronoids, including a general characterization
of their support functions, descriptions of several classes of metronoids,
and the various properties of metronoids generated by discrete measures.
In Section \ref{sec:Estimating_d_D}, we prove Theorems \ref{thm:main_upper_sym}
and \ref{thm:main_upper_bdd}. In Section \ref{sec:fvein}, we discuss
the fractional vertex index, provide precise computations of the fractional
vertex index of $B_{1}^{n}$ and $B_{2}^{n}$, and prove Theorem \ref{thm:fvein_lower_bdd}.
We conclude this paper with a proof of Corollary \ref{cor:centrod_low_bdd}
in Section \ref{sec:app_centroid}.\\

\noindent \textbf{Acknowledgements.} The authors thank Alon Nishry
and Beatrice Vritsiou for useful discussions. The second named author
thanks Shiri Artstein-Avidan for helpful conversations on possible
extensions of the vertex index, and for her comments regarding the
written text. 

\section{Properties of Metronoids\label{sec:Met_prop}}

\subsection{Descriptions of Metronoids}

In this section we give several geometric descriptions of metronoids
in terms of their generating measures. 

\subsubsection{\textbf{A general characterization}}

Let $\mu$ be any finite Borel measure on $\mathbb{R}^{n}$. We begin
with providing a formula for the support function of $\met{\mu}$.

For each $\theta\in\Sph^{n-1}$, define 
\begin{equation}
R\left(\theta\right):=\max\{R\in\R\,,\,\mu\left\{ \left(x\in\mathbb{R}^{n}\,:\,\iprod x{\theta}\ge R\}\right)\ge1\right\} .\label{eq:R_theta}
\end{equation}
 Correspondingly, we define $f_{\theta}:\R\rightarrow[0,1]$ as follows:
\begin{eqnarray}
f_{\theta}(t)=\left\{ \begin{array}{cc}
0, & \text{ if \ensuremath{t<R},}\\
1, & \text{ if \ensuremath{t>R},}\\
0, & \text{ \ensuremath{t=R} and \ensuremath{\mu\left(\left\{ \iprod x{\theta}=R\left(\theta\right)\right\} \right)=0},}\\
\frac{1-\mu\left(\left\{ \iprod x{\theta}>R(\theta)\right\} \right)}{\mu\left(\left\{ \langle x,\theta\rangle=R(\theta)\right\} \right)}, & \text{ \ensuremath{t=R} and \ensuremath{\mu\left(\left\{ \iprod x{\theta}=R\left(\theta\right)\right\} \right)\neq0}}.
\end{array}\right.\label{eq:f_theta}
\end{eqnarray}
One can easily verify that $0\le f_{\theta}\le1$, and $\int_{\R^{n}}f_{\theta}\left(\iprod x{\theta}\right))\d{\mu}\left(x\right)=1$.
Therefore, 

\[
y_{\theta}:=\int_{\R^{n}}xf_{\theta}\left(\iprod x{\theta}\right)\d{\mu}\left(x\right)\in\met{\mu}.
\]
 The following proposition describes the support function of $\met{\mu}$
in direction $\theta$, in terms of $y_{\theta}$:
\begin{prop}
\label{ExtremePointofMu} With the notation above, for any $y\in\met{\mu}$,
and $\theta\in\Sph^{n-1}$, $\langle y,\theta\rangle\le\langle y_{\theta},\theta\rangle$.
Namely, $h_{\met{\mu}}(\theta)=\langle y_{\theta},\theta\rangle$.
\end{prop}

\begin{proof}
Fix $\theta\in\Sph^{n-1}$ and let $y\in\met{\mu}$. Then there exists
a function $0\le f(x)\le1$ such that $\int_{\R^{n}}f(x)\d{\mu}\left(x\right)=1$
and $\int_{\R^{n}}xf(x)\d{\mu}\left(x\right)=y$. Then, denoting $R=R\left(\theta\right)$,
we have:
\begin{eqnarray*}
\iprod{y_{\theta}}{\theta}-\iprod y{\theta} & = & \int_{\R^{n}}f_{\theta}\left(\iprod x{\theta}\right)\iprod x{\theta}\d{\mu}\left(x\right)-\int_{\R^{n}}f(x)\iprod x{\theta}\d{\mu}\left(x\right)\\
 & = & \int_{\R^{n}}\left(f_{\theta}\left(\iprod x{\theta}\right)-f(x)\right)\iprod x{\theta}\d{\mu}\left(x\right)\\
 & = & \int_{\langle x,\theta\rangle>R}\left(f_{\theta}\left(\iprod x{\theta}\right)-f(x)\right)\iprod x{\theta}\d{\mu}\left(x\right)\\
 & + & \int_{\langle x,\theta\rangle<R}\left(f_{\theta}\left(\iprod x{\theta}\right)-f(x)\right)\iprod x{\theta}\d{\mu}\left(x\right)\\
 & + & \int_{\langle x,\theta\rangle=R}\left(f_{\theta}\left(\iprod x{\theta}\right)-f(x)\right)R\d{\mu}\left(x\right).
\end{eqnarray*}
By the definition of $f_{\theta},$ it follows that $f_{\theta}\left(\iprod x{\theta}\right)-f(x)\ge0$
whenever $\langle x,\theta\rangle>R$, and $f_{\theta}\left(\iprod x{\theta}\right)-f(x)\le0$
whenever $\langle x,\theta\rangle<R$. Therefore, we have that for
every $x\in\R^{n}$, $\left(f_{\theta}\left(\iprod x{\theta}\right)-f(x)\right)\iprod x{\theta}\ge\left(f_{\theta}\left(\iprod x{\theta}\right)-f(x)\right)R,$
which together which the above equality implies that
\[
\iprod{y_{\theta}}{\theta}-\iprod y{\theta}\ge\int_{\R^{n}}\left(f_{\theta}\left(\iprod x{\theta}\right)-f(x)\right)R\d{\mu}\left(x\right)=0.
\]
\end{proof}
For each $\theta\in\Sph^{n-1}$, define $H_{\theta}^{+}:=\left\{ x\in\R^{n}\,:\,\iprod x{\theta}>0\right\} $.
In the sequel, we will also need the following useful fact:
\begin{prop}
\label{prop:necessary_1}Let $K\sub\R^{n}$ be a convex body. Suppose
$\mu$ is a measure such that $K\sub\met{\mu}$. Then for every $\theta\in\Sph^{n-1}$
we have that 
\[
h_{K}\left(\theta\right)\le\int_{H_{\theta}^{+}}\iprod x{\theta}\d{\mu}\left(x\right).
\]
\end{prop}

\begin{proof}
Fix $\theta\in\Sph^{n-1}$, and let $x_{\theta}\in K$ such that
$h_{K}\left(\theta\right)=\iprod{x_{\theta}}{\theta}$. Since $K\sub\met{\mu}$,
there exists a function $0\le f\le1$ such that $x_{\theta}=\int_{\R^{n}}xf\left(x\right)\d{\mu}\left(x\right)$,
and hence 
\begin{align*}
h_{K}\left(\theta\right) & =\int_{\R^{n}}\iprod x{\theta}f\left(x\right)\d{\mu}\left(x\right)\le\int_{H_{\theta}^{+}}\iprod x{\theta}f\left(x\right)\d{\mu}\left(x\right)\le\int_{H_{\theta}^{+}}\iprod x{\theta}\d{\mu}\left(x\right).
\end{align*}
\end{proof}

\subsubsection{\textbf{Discrete measures}}

In this section we provide some geometric description of metronoids
that are generated by discrete measures. 

The first property states that the metronoid generated by a finite
discrete measure is a polytope: 
\begin{prop}
Let $x_{1},\dots,x_{m}\in\R^{n}$, $w_{1},\dots,w_{m}>0$, and define
$\mu=\sum_{i=1}^{m}w_{i}\delta_{x_{i}}$. Then $\met{\mu}$ is a polytope.
\end{prop}

\begin{proof}
Consider the linear map $F:\R^{m}\to\R^{n}$ defined by $F\left(\left(\lambda_{1},\dots,\lambda_{m}\right)\right)=\sum_{i=1}^{m}\lambda_{i}w_{i}x_{i}$,
and consider the polytope $P=\left\{ \left(\lambda_{1},\dots,\lambda_{m}\right)\in\R^{m}\,:\,0\le\lambda_{1},\dots,\lambda_{m}\le1,\,\sum_{i=1}^{m}\lambda_{i}w_{i}=1\right\} $.
Then, by definition, $\met{\mu}=F\left(P\right)$, and hence a polytope
as well. 
\end{proof}
For our next observation we need the following notation. Given $x_{1},\dots,x_{m}\in\R^{n}$,
denote the Minkowski sum of the segments $\left\{ \left[0,x_{i}\right]\right\} _{i=1}^{m}$
by $Z\left(x_{1},\dots,x_{m}\right)$. That is, 
\[
Z\left(x_{1},\dots,x_{m}\right)=\left\{ \lambda_{1}x_{1}+\dots+\lambda_{m}x_{m}\,:\,\lambda_{1},\dots,\lambda_{m}\in\left[0,1\right]\right\} .
\]
In the following proposition, we show that given a measure $\mu=\sum_{i=1}^{m}w_{i}\delta_{x_{i}}$,
its generated metronoid is always contained in the intersection of
$\conv{x_{1},\dots x_{m}}$ and the zonotope $Z\left(w_{1}x_{1},\dots,w_{m}x_{m}\right)$.
\begin{prop}
\label{prop:Z_cap_P}Let $x_{1},\dots x_{m}\in\R^{n}$, $w_{1},\dots,w_{m}>0$,
and set $\mu=\sum_{i=1}^{m}w_{i}\delta_{x_{i}}$. Then 
\[
\met{\mu}\sub\conv{x_{1},\cdots,x_{m}}\cap Z\left(w_{1}x_{1},\dots,w_{m}x_{m}\right).
\]
\end{prop}

\begin{proof}
Recall that $\met{\mu}=\left\{ \sum_{i=1}^{m}\lambda_{i}w_{i}x_{i}\,:\,0\le\lambda_{i}\le1\:,\:\sum_{i=1}^{m}\lambda_{i}w_{i}=1\right\} $.
Then, on the one hand, we may relax the first constraint and obtain
that 
\[
\met{\mu}\sub P:=\left\{ \sum_{i=1}^{m}\lambda_{i}w_{i}x_{i}\,:\:\lambda_{i}\ge0,\,\sum_{i=1}^{m}\lambda_{i}w_{i}=1\right\} =\conv{x_{1},\dots,x_{m}}.
\]
On the other hand, we may remove the second constraint and obtain
that 
\[
\met{\mu}\sub Z:=Z\left(w_{1}x_{1},\dots w_{m}x_{m}\right)=\left\{ \sum_{i=1}^{m}\lambda_{i}w_{i}x_{i}\,:\,0\le\lambda_{i}\le1\right\} .
\]
Therefore, we clearly have that $\met{\mu}\sub P\cap Z$. 
\end{proof}
A picture demonstrating Proposition \ref{prop:Z_cap_P} is given in
Figure \ref{fig:discrete_met} below in the particular case where
$\mu=\delta_{0}+\sum_{i=1}^{2}\frac{1}{k}\left(\delta_{e_{i}}+\delta_{-e_{i}}\right)$. 

\begin{figure}[H]
\begin{centering}
\includegraphics[scale=0.6]{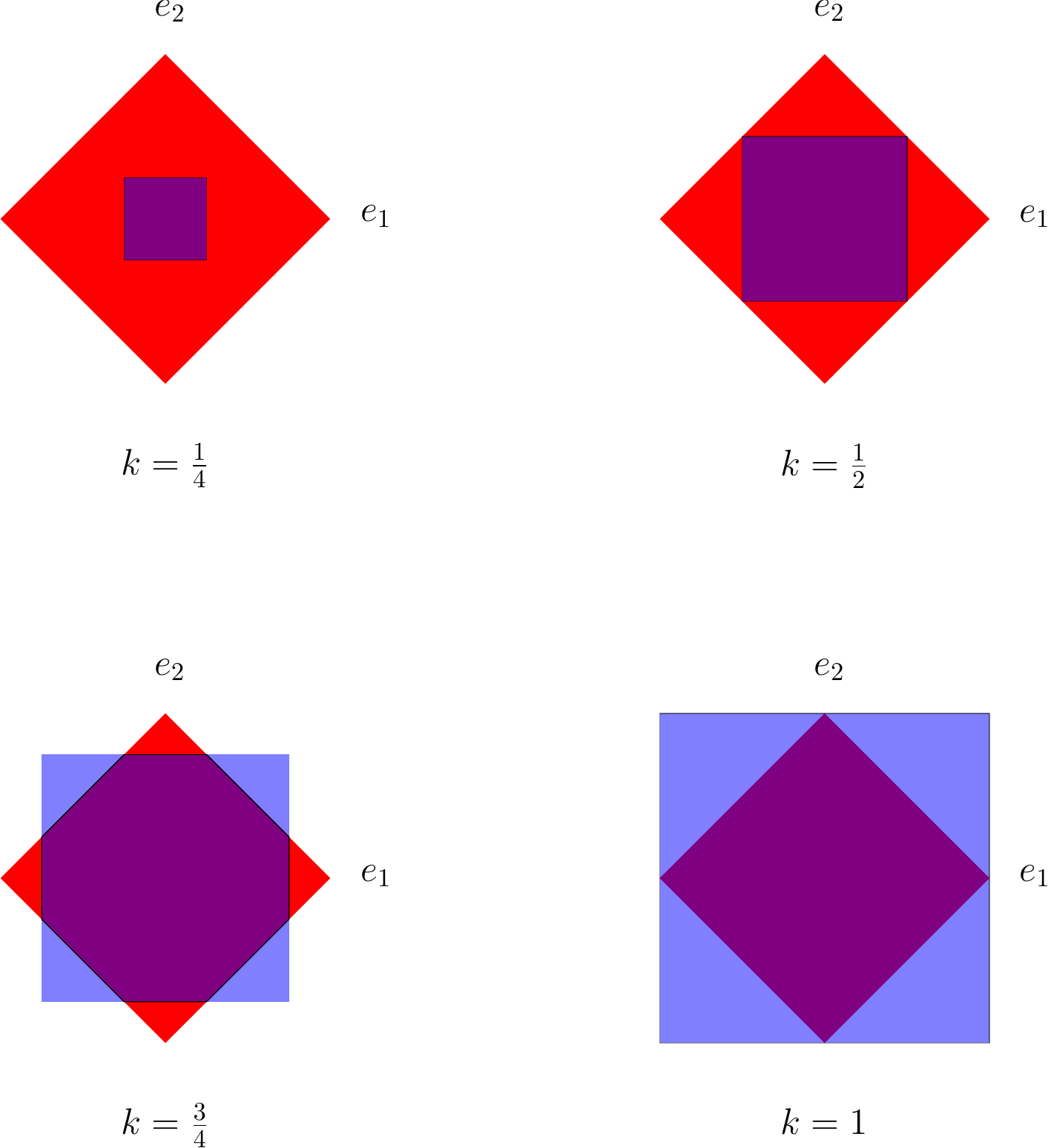}
\par\end{centering}
\caption{\label{fig:discrete_met}The metronoid generated by $\mu=\delta_{0}+\sum_{i=1}^{2}\frac{1}{k}\left(\delta_{e_{i}}+\delta_{-e_{i}}\right)$
for different values of $k$. Here $\protect\conv{\pm\frac{e_{1}}{k},\pm\frac{e_{2}}{k}}$
is marked in red, $Z\left(\pm\frac{e_{1}}{k},\pm\frac{e_{2}}{k}\right)$
in blue, and $\protect\met{\mu}$ in purple. }
\end{figure}

We remark that in Figure \ref{fig:discrete_met}, we have that $\met{\mu}=Z\left(\mu\right)\cap P\left(\mu\right)$
for all values of $k$. However, this is not always the case. For
example, consider $\mu=\sum_{i=1}^{2}\frac{1}{4}\left(\delta_{e_{i}}+\delta_{e_{-i}}\right)$
on $\R^{2}$. Then $\mu\left(\R^{2}\right)=1$, and hence $\met{\mu}=\left\{ 0\right\} \neq Z\left(\mu\right)\cap P\left(\mu\right)$.

\subsubsection{\textbf{Zonoid generating measures }}

Proposition \ref{prop:Z_cap_P} can be stated in a more general case.
Given a Borel measure $\mu$ on $\R^{n}$, define 
\[
Z\left(\mu\right)=\left\{ \int_{\R^{n}}xf\left(x\right)\d{\mu}\left(x\right)\,:\,0\le f\le1\right\} ,\,P\left(\mu\right)=\left\{ \int_{\R^{n}}xf\left(x\right)\d{\mu}\left(x\right)\,:\,0\le f,\,\int_{\R^{n}}f\d{\mu}=1\right\} .
\]
Then, the same argument verbatim as in the proof of Proposition \ref{prop:mu_Z_cap_p}
yields: 
\begin{prop}
\label{prop:mu_Z_cap_p}We have that $\met{\mu}\sub Z\left(\mu\right)\cap P\left(\mu\right)$. 
\end{prop}

\begin{rem}
\label{rem:fvein_zonotope}To complement Proposition \ref{prop:mu_Z_cap_p},
let $\mu$ be a finite Borel measure satisfying that $\mu\left(\R^{n}\right)\le2$
and $\mu\left(\left\{ 0\right\} \right)\ge1$. We claim that in this
case $\met{\mu}=Z\left(\mu\right)$. Indeed, note that for any function
$0\le f\le1$, $\int_{\R^{n}}f\left(x\right)\d{\mu}\left(x\right)\le\mu\left(\left\{ 0\right\} \right)f\left(0\right)+1$.
Hence, by changing the value of $f\left(0\right)$ (which does not
affect \textbf{$\int_{\R^{n}}xf\left(x\right)\d{\mu}\left(x\right)$),}
we may assume that $\int_{\R^{n}}f\left(x\right)\d{\mu}\left(x\right)=1$.
Therefore, it follows that, under these assumptions, $\met{\mu}=Z\left(\mu\right)$.
This fact is also demonstrated in Figure \ref{fig:discrete_met} above,
for $\mu=\delta_{0}+\sum_{i=1}^{2}\frac{1}{k}\left(\delta_{e_{i}}+\delta_{-e_{i}}\right)$
and $\frac{1}{4}\le k\le\frac{1}{2}$. 
\end{rem}

The next proposition shows that by adding symmetricity to the measures
described in Remark \ref{rem:fvein_zonotope}, the generated metronoids
become zonoids:
\begin{prop}
\label{prop:MuCentroid-1} Suppose $\mu$ is a symmetric Borel measure
satisfying that $\mu(\R^{n})\le2$ , and $\mu\left(\left\{ 0\right\} \right)\ge1$.
Then
\[
h_{\met{\mu}}(\theta)=\frac{1}{2}\int_{\R^{n}}\left|\iprod{\theta}x\right|\d{\mu}\left(x\right).
\]
\end{prop}

\begin{proof}
Fix $\theta\in\Sph^{n-1}.$ Recall the definition of $R_{\theta}$
and $f_{\theta}$ in \eqref{eq:R_theta}, and \eqref{eq:f_theta}.
Observe that since $\mu$ is symmetric, $\mu\left(\R^{n}\right)\le2$,
and $\mu\left(\left\{ 0\right\} \right)\ge1$, it follows that $R_{\theta}=0$.
Therefore, Proposition \ref{ExtremePointofMu} implies that 
\[
h_{\met{\mu}}\left(\theta\right)=\int_{\iprod x{\theta}\ge0}f_{\theta}\left(\iprod x{\theta}\right)\iprod x{\theta}\d{\mu}\left(x\right)=\int_{\iprod x{\theta}>0}\iprod x{\theta}\d{\mu}\left(x\right)=\frac{1}{2}\int_{\R^{n}}\left|\iprod x{\theta}\right|\d{\mu}\left(x\right).
\]
\end{proof}

\subsubsection{\textbf{Uniform measures on convex bodies}}

Let $K\sub\R^{n}$ be a convex body, and fix $0<\delta<\vol K$. Let
$\mu_{\delta}$ be the uniform measure on $K$, defined by $\d{\mu_{\delta}=\delta^{-1}\one_{K}\d x}$.
Then, for any direction $\theta\in\Sph^{n-1}$, Proposition \ref{ExtremePointofMu}
tells us that $h_{\met{\mu_{\delta}}}\left(\theta\right)=\iprod{y_{\theta}}{\theta}$
where 
\[
y_{\theta}=\frac{1}{\delta}\int_{\R^{n}}x\one_{\left\{ y\in K\,:\,\iprod y{\theta}\ge R\left(\theta\right)\right\} }\left(x\right)\d x
\]
and $R\left(\theta\right)$ is the real number satisfying that $\vol{\left\{ x\in K\,:\,\iprod x{\theta}\ge R\left(\theta\right)\}\right\} }=\delta.$ 

The body $\met{\mu}$ is related to the floating body $K_{\delta}=\bigcap_{\theta\in\Sph^{n-1}}\left\{ x\in\R^{n}\,:\,\iprod x{\theta}\le R\left(\theta\right)\right\} $
in the following sense: the boundary points of $\met{\mu}$ are the
centers of mass of the caps $\left\{ x\in K:\,\iprod x{\theta}\ge R\left(\theta\right)\right\} $
which are cut off in order to obtain $K_{\delta}$ (see \cite{SW90}
for more about floating bodies). In fact, one can show that $K_{\delta}\sub\met{\mu_{\delta}}\sub K_{\frac{\delta}{e}}$. 

\subsection{Some linear-invariance properties}

In this section we state a few basic facts concerning the behavior
of metronoids under linear transformations, and the invariance of
the quantities $d_{R}^{*}\left(K\right)$, $D_{R}^{*}\left(K\right)$,
and $\fvein K$. 

Let $T\in{\rm GL}_{n}\left(\R\right)$ be an invertible linear transformation
on $\R^{n}$. Given a Borel measure $\mu$ on $\R^{n}$, denote by
$\nu=T\#\mu$ the pushforward of $\mu$ by $T$, that is $\nu\left(A\right)=\mu\left(T^{-1}A\right)$
for any Borel set $A\sub\R^{n}$. Then we have:
\begin{fact}
\label{fact:met_lin_inv} Let $\mu$ be a Borel measure on $\R^{n}$,
$T\in{\rm GL}_{n}\left(\R\right)$, and denote $\nu=T\#\mu$. Then
$\met{\nu}=T\met{\mu}$. Moreover, for any convex body $K\sub\R^{n}$
containing the origin in its interior, we have that $\int_{\R^{n}}\norm x_{K}\d{\mu}\left(x\right)=\int_{\R^{n}}\norm x_{TK}\d{\nu}\left(x\right)$. 
\end{fact}

\begin{proof}
Let $x\in\met{\mu}$. Then $x=\int_{\R^{n}}yf\left(y\right)\d{\mu}\left(y\right)$
for some $0\le f\le1$ with $\int_{\R^{n}}f\d{\mu}=1$, and hence
\[
Tx=\int_{\R^{n}}Tyf\left(y\right)\d{\mu}\left(y\right){\it =\int_{\R^{n}}yf\left(T^{-1}y\right)\d{\nu}\left(y\right)}\in\met{\nu}.
\]
Similarly, if $z\in\met{\nu}$ then $z=\int_{\R^{n}}yg\left(y\right)\d{\nu}\left(y\right)$
for some $0\le g\le1$ with $\int_{\R^{n}}g\d{\nu}=1$, and hence
\[
z=\int_{\R^{n}}yg\left(y\right)\d{\nu}\left(y\right){\it =\int_{\R^{n}}Tyg\left(Ty\right)\d{\mu}\left(y\right)}\in T\met{\mu}.
\]
Let $K\sub\R^{n}$ be a convex body containing the origin in its interior.
Then 
\[
\int_{\R^{n}}\norm x_{TK}\d{\nu}\left(x\right)=\int_{\R^{n}}\norm{Tx}_{TK}\d{\mu}\left(x\right)=\int_{\R^{n}}\norm x_{K}\d{\mu}\left(x\right).
\]
\end{proof}
\begin{fact}
\label{fact:Linear_invariant}Let $K$ be a convex body in $\R^{n}$,
$T\in{\rm GL}_{n}\left(\R\right)$, and $R\ge1$. Then $d_{R}^{*}\left(K\right)=d_{R}^{*}\left(TK\right)$,
$D_{R}^{*}\left(K\right)=D_{R}^{*}\left(TK\right)$, and $\fvein K=\fvein{TK}$. 
\end{fact}

\begin{proof}
Let $\mu$ be a measure such that $K\sub\met{\mu}\sub RK$, and let
$T\in{\rm GL}_{n}\left(\R\right)$. Then by considering the pushforward
measure $\nu=T\#\mu$. By Fact \ref{fact:met_lin_inv}, we have that
$\met{\nu}=T\met{\nu}$, and hence $TK\sub\met{\nu}\sub R\left(TK\right)$.
Moreover, we clearly have that $\nu\left(\R^{n}\right)=\mu\left(\R^{n}\right)$,
from which it follows that $d_{R}^{*}\left(K\right)=d_{R}^{*}\left(TK\right)$.
Finally, note that Fact \ref{fact:met_lin_inv} also implies that
$\int_{\R^{n}}\norm x_{TK}\d{\nu}\left(x\right)=\int_{\R^{n}}\norm x_{K}\d{\mu}\left(x\right)$,
which means that $D_{R}^{*}\left(K\right)=D_{R}^{*}\left(TK\right)$,
as required.
\end{proof}

\subsection{Approximations by discrete measures}

In this section we show that, for the purpose of approximating a
convex body $K\sub\R^{n}$, one can often replace a general Borel
measure $\mu$ by a finite discrete measure, without increasing the
cost $\int_{\R^{n}}\norm x_{K}\d{\mu}\left(x\right)$. 

We begin with the reduction of infinite measures to finite measures: 
\begin{lem}
Let $K\sub\R^{n}$ \label{lem:approx_infinite} be a convex body
containing $0$ in its interior, and $\mu$ be an infinite Borel measure
such that $K\sub\met{\mu}$, and $\int_{\R^{n}}\norm x_{K}\d{\mu}\left(x\right)<\infty$.
Then for any $\eps>0$, there exists a finite Borel measure $\nu$
such that $\met{\mu}\sub\met{\nu}\sub\left(1+\eps\right)\met{\mu}$
and, in particular, $K\sub\met{\nu}$. Furthermore, we also have that
$\int_{\R^{n}}\norm x_{K}\d{\nu}\left(x\right)\le(1+\eps)\int_{\R^{n}}\norm x_{K}\d{\mu}\left(x\right)$. 
\end{lem}

\begin{proof}
First, we show that we can reduce to the case where $\mu\left(\left\{ 0\right\} \right)<\infty$.
Indeed, suppose $\mu\left(\left\{ 0\right\} \right)=\infty$, and
define a measure $\nu$ by setting $\nu\left(A\right)=\mu\left(A\backslash\left\{ 0\right\} \right)$
for any measurable set $A$. Let $y\in\met{\mu},$ and $0\le f\le1$
be a function such that $\int_{\R^{n}}f\left(x\right)\d{\mu\left(x\right)}=1$
and $y=\int_{\R^{n}}xf\left(x\right)\d{\mu\left(x\right)}$. The conditions
$\mu\left(\left\{ 0\right\} \right)=\infty$ and $\int_{\R^{n}}f\left(x\right)\d{\mu\left(x\right)}=1$
force $f\left(0\right)=0$. Thus, $\int_{\R^{n}}f\left(x\right)\d{\nu\left(x\right)=1}$
and $y=\int_{\R^{n}}xf\left(x\right)\d{\nu\left(x\right)}$, which
implies that $y\in\met{\nu}$. On the other hand, let $y'\in\met{\nu}$,
and $0\le f\le1$ be a function satisfying that $\int_{\R^{n}}f\left(x\right)\d{\nu\left(x\right)}=1$
and $y'=\int_{\R^{n}}xf\left(x\right)\d{\nu\left(x\right)}$. Since
$\left\{ 0\right\} $ is not in the support of $\nu$, we may assume
without loss of generality that $f\left(0\right)=0$. Hence, $\int_{\R^{n}}f\left(x\right)\d{\mu\left(x\right)=1}$
and $y'=\int_{\R^{n}}xf\left(x\right)\d{\mu\left(x\right)}$, which
implies that $y'\in\met{\mu}$. Furthermore, we have that $\int_{\R^{n}}\norm x_{K}\d{\nu}\left(x\right)=\int_{\R^{n}}\norm x_{K}\d{\mu}\left(x\right)$.
Thus, from now on we may assume that $\mu\left(\left\{ 0\right\} \right)<\infty$. 

By Fact \ref{fact:met_lin_inv}, for any $T\in GL_{n}\left(\R\right)$,
we have that $K\sub\met{\mu}\iff TK\sub\met{T\#\mu}$, and $\int_{\R^{n}}\norm x{}_{K}\d{\mu}\left(x\right)=\int_{\R^{n}}\norm x{}_{TK}\d{\left(T\#\mu\right)}\left(x\right)$.
Therefore, we may assume without loss of generality that $B_{2}^{n}\sub K\sub\met{\mu}$.

Define the measure $\nu$ by: 

\[
\d{\nu}=\delta_{0}+1_{(\lambda B_{2}^{n})^{c}}\d{\mu},
\]
where $\lambda>0$ is a parameter that will be determined later. Since
$K\sub CB_{2}^{n}$ for some $C>0$, we have that 
\[
\mu\left(\left(\lambda B_{2}^{n}\right){}^{c}\right)=\int_{\left(\lambda B_{2}^{n}\right)^{c}}1\d{\mu}\le\frac{1}{\lambda}\int_{(\lambda B_{2}^{n})^{c}}\norm x_{B_{2}^{n}}\d{\mu}\left(x\right)\le\frac{C}{\lambda}\int_{\R^{n}}\norm x_{K}\d{\mu}\left(x\right)<\infty.
\]
Hence we have that $\nu$ is a finite measure, and $\mu$$\left(\lambda B_{2}^{n}\right)=\infty$.

Let $y\in\met{\mu}$. Then there exists a function $0\le f\le1$ such
that $\int_{\R^{n}}f\d{\mu}=1$ and $y=\int_{\R^{n}}xf\left(x\right)\d{\mu}\left(x\right)$.
Let $a=1-\int_{(\lambda B_{2}^{n})^{c}}f(x)\d{\mu}(x)$, and define
the function:
\[
g(x)=\begin{cases}
f(x), & x\in(\lambda B_{2}^{n})^{c}\\
a, & x=0\\
0, & {\rm otherwise}
\end{cases}.
\]
Then $0\le g\le1$ and $\int_{\R^{n}}g(x)\d{\nu}(x)=1$. Denoting
$y'=\int xg(x)\d{\nu}(x)$, we have that 

\begin{align*}
\norm{y-y'}_{B_{2}^{n}} & =\norm{\int_{\R^{n}}xf(x)\d{\mu}(x)-\int_{\R^{n}}xg(x)\d{\nu}(x)}_{B_{2}^{n}}\\
 & =\norm{\int_{\lambda B_{2}^{n}}xf(x)\d{\mu}(x)}_{B_{2}^{n}}\\
 & \le\lambda\int_{\lambda B_{2}^{n}}f(x)\d{\mu}(x)\\
 & \le\lambda.
\end{align*}
Similarly, for any $y'\in\met{\nu}$ there exists $y\in\met{\mu}$
such that $\|y-y'\|_{B_{2}^{n}}\le\lambda$. Indeed, let $0\le g\le1$
be a function such that $\int_{\R^{n}}g\left(x\right)\d{\nu\left(x\right)}=1$
and $y'=\int_{\R^{n}}xg\left(x\right)\d{\nu\left(x\right)}$. To define
a corresponding function $f\left(x\right)$, fix some $0<s<\lambda$
so that $1\le\mu\left(\left\{ x\,:\,s\le\norm x_{B_{2}^{n}}<\lambda\right\} \right)<\infty$.
The second inequality holds for any $s>0$. If there is no $s>0$
such that the first inequality is satisfied, then $\mu\left(\left\{ x\,:\,0<\norm x_{B_{2}^{n}}\le\lambda\right\} \right)\le1$,
which together with $\mu\left(\left\{ 0\right\} \right)<\infty$,
contradicts the fact that $\mu$$\left(\lambda B_{2}^{n}\right)=\infty$.
Define 
\[
f\left(x\right)=\begin{cases}
g\left(x\right), & x\in(\lambda B_{2}^{n})^{c}\\
\text{\ensuremath{\frac{1-\int_{(\lambda B_{2}^{n})^{c}}f(x)\d{\nu}(x)}{\mu\left(\left\{ x\,:\,s\le\norm x_{B_{2}^{n}}<\lambda\right\} \right)}} }, & s\le\norm x_{B_{2}^{n}}<\lambda\\
0, & {\rm otherwise}
\end{cases}.
\]
Since $0\le1-\int_{(\lambda B_{2}^{n})^{c}}g(x)\d{\nu}(x)\le1$, it
follows that $\text{0\ensuremath{\le\frac{1-\int_{(\lambda B_{2}^{n})^{c}}f(x)\d{\nu}(x)}{\mu\left(\left\{ x\,:\,s\le\norm x_{B_{2}^{n}}<\lambda\right\} \right)}} }\le1$,
and hence $0\le f\le1$. Moreover,
\begin{align*}
\int_{\R^{n}}f\left(x\right)\d{\mu\left(x\right)} & =\int_{(\lambda B_{2}^{n})^{c}}f(x)\d{\nu}(x)+\int_{\left\{ y\,:\,s\le\norm y_{B_{2}^{n}}<\lambda\right\} }\frac{1-\int_{(\lambda B_{2}^{n})^{c}}f(y)\d{\nu}(y)}{\mu\left(\left\{ y\,:\,s\le\norm y_{B_{2}^{n}}<\lambda\right\} \right)}\d{\mu\left(x\right)}=1.
\end{align*}
Denoting $y=\int_{\R^{n}}xf\left(x\right)\d{\mu\left(x\right)}\in\text{\ensuremath{\met{\mu}}}$,
it follows that 
\[
\norm{y-y'}_{B_{2}^{n}}=\norm{\int_{\R^{n}}xf(x)\d{\mu}(x)-\int_{\R^{n}}xg(x)\d{\nu}(x)}_{B_{2}^{n}}=\norm{\int_{\lambda B_{2}^{n}}xf(x)\d{\mu}(x)}_{B_{2}^{n}}\le\lambda.
\]
To show the inclusion $\left(1-\lambda\right)\met{\mu}\sub\met{\nu}$,
let 
\[
\met{\mu}^{\circ}=\left\{ x\in\R^{n}\,:\,\iprod xy\le1\,\forall\,y\in\met{\mu}\right\} 
\]
denote the polar body of $\met{\mu}$. By the properties of polarity,
for any $z\in\met{\mu}^{\circ}$, there exists $y\in M\left(\mu\right)$
such that $\iprod zy=1$. Moreover, by the previous argument, there
exists $y'\in\met{\nu}$ such that $\norm{y-y'}_{B_{2}^{n}}\le\lambda$.
Therefore, we have that 
\begin{align*}
\iprod z{y'} & =\iprod zy-\iprod z{y-y'}\\
 & \ge1-\norm z_{2}\norm{y-y'}_{2}\\
 & \ge1-\lambda,
\end{align*}
where the last inequality is due to the fact that $\norm z_{2}\le1$,
as $\met{\mu}^{\circ}\sub B_{2}^{n}$. Thus, it follows that $\left(1-\lambda\right)\met{\mu}\sub\met{\nu}$. 

For the opposite inclusion, we use the fact that for every $y'\in\met{\nu}$,
there exists $y\in\met{\mu}$ such that $\norm{y-y'}_{B_{2}^{n}}\le\lambda$.
Equivalently, $\met{\nu}\sub\met{\mu}+\lambda B_{2}^{n}$. Since $B_{2}^{n}\sub\met{\mu}$,
it follows that $\met{\nu}\sub\left(1+\lambda\right)\met{\mu}$, and
so 
\[
\left(1-\lambda\right)\met{\mu}\sub\met{\nu}\sub\left(1+\lambda\right)\met{\mu}.
\]
Moreover, by the definition of $\nu$ we have that 
\[
\int_{\R^{n}}\norm x{}_{(1-\lambda)K}\d{\nu}(x)\le\int_{\R^{n}}\norm x{}_{(1-\lambda)K}\d{\mu}(x)=\left(1-\lambda\right)^{-1}\int_{\R^{n}}\norm x_{K}\d{\mu}\left(x\right).
\]

Finally, consider the pushforward measure $\widetilde{\nu}=\frac{I_{n}}{1-\lambda}\#\nu$.
By Fact \ref{fact:met_lin_inv}, we have that 
\[
\met{\mu}\sub\met{\widetilde{\nu}}\sub\frac{1+\lambda}{1-\lambda}\met{\mu}
\]
and, in particular, $K\sub\met{\tilde{\nu}}$. Furthermore, we have
that 
\[
\int_{\R^{n}}\norm x_{K}\d{\widetilde{\nu}}\left(x\right)=\left(1-\lambda\right)^{-1}\int_{\R^{n}}\norm x{}_{K}\d{\nu}(x)\le\left(1-\lambda\right)^{-1}\int_{\R^{n}}\norm x_{K}\d{\mu}\left(x\right).
\]
By choosing a sufficiently small $\lambda$, the proof is complete. 
\end{proof}
The next lemma shows that any finite measure can be replaced with
a discrete one: 
\begin{lem}
\label{lem:appx_discrete}Let $K\sub\R^{n}$ be a convex body containing
$0$ in its interior, and $\mu$ be a finite Borel measure such that
$K\sub\met{\mu}$, and $\int_{\R^{n}}\norm x_{K}\d{\mu}\left(x\right)<\infty$.
Then for any $\eps>0$, there exists a finite discrete measure $\nu$
such that $\met{\mu}\sub\met{\nu}\sub\frac{1+2\eps}{1-2\eps}\met{\mu}$
and, in particular, $K\sub\met{\nu}$. Moreover, 
\[
\int_{\mathbb{R}^{n}}\norm x{}_{K}\d{\nu}(x)\le\frac{1}{1-2\eps}\int_{\mathbb{R}^{n}}\norm x{}_{K}\d{\mu}(x)+\frac{\eps}{1-2\eps}.
\]
\end{lem}

\begin{proof}
As in the proof of the previous lemma, we may assume without loss
of generality that $B_{2}^{n}\sub K$. Fix $\eps>0$, and fix some
large $R\in\mathbb{N}$ so that $\int_{\left(RB_{\infty}^{n}\right)^{c}}\norm x_{K}\d{\mu}\left(x\right)\le\eps$,
where $B_{\infty}^{n}$ denotes the $n$-cube $\left[-1,1\right]^{n}\sub\R^{n}$.
For any $m\in\mathbb{N}$, let $A_{m}\subset\mathbb{R}^{n}$ be the
collection of points 

\[
A_{m}:=\left\{ (a_{1},\cdots,a_{n})\,:\,\forall i\:a_{i}\in\left\{ 0,\pm\frac{1}{2^{m}},\pm\frac{2}{2^{m}}\dots,\pm R\right\} \right\} .
\]
For each $a\in A_{m}$, we define the box $B_{a}$ by
\[
B_{a}:=\left\{ (x_{1},\cdots,x_{n})\,:\,\forall i\,\,a_{i}-\frac{1}{2^{m+1}}\le x_{i}<a_{i}+\frac{1}{2^{m+1}}\right\} ,
\]
and observe that $\{B_{a}\}_{a\in A_{m}}$ is a partition of $E:=\left[-\left(R+\frac{1}{2^{m+1}}\right),\left(R+\frac{1}{2^{m+1}}\right)\right)^{n}$.
Fix a large enough $m$ so that for each $a\in A_{m}$ and every $x,y\in B_{a}$,
we have that $\norm{x-y}{}_{B_{2}^{n}}<\frac{\eps}{\mu\left(\R^{n}\right)}<\eps$.
Define the measure 
\[
\mu_{m}:=\sum_{a\in A_{m}\setminus\left\{ 0\right\} }\mu(B_{a})\delta_{a}+\left(\mu\left(B_{0}\right)+\mu\left(E^{c}\right)\right)\delta_{0}.
\]
We claim that for every $y\in\met{\mu}$, there exists $y'\in\met{\mu_{m}}$
such that $\norm{y-y'}_{B_{2}^{n}}\le2\eps$. Indeed, let $y\in\met{\mu}$,
and $0\le f\le1$ be a function such that $\int_{\R^{n}}f\left(x\right)\d{\mu}\left(x\right)=1$
and $y=\int_{\R^{n}}xf\left(x\right)\d{\mu}\left(x\right)$. Correspondingly,
we define the function $g$ with support in $A_{m}$ by setting 
\[
g(a):=\begin{cases}
\frac{\int_{B_{a}}f(x)\d{\mu}(x)}{\mu(B_{a})}, & \text{ }a\in A_{m}\setminus\left\{ 0\right\} \text{ and }\mu\left(B_{a}\right)\neq0\\
\frac{\int_{B_{0}}f\left(x\right)\d{\mu}\left(x\right)+\int_{E^{c}}f\left(x\right)\d{\mu}\left(x\right)}{\mu\left(B_{0}\right)+\mu\left(E^{c}\right)}, & \text{ }a=0\text{ and }\mu\left(B_{0}\right)+\mu\left(E^{c}\right)\neq0\\
0, & \text{otherwise}
\end{cases}.
\]
 One can verify that $0\le g\le1$. Moreover, we have that 
\begin{align*}
\int_{\mathbb{R}^{n}}g(x)\d{\mu_{m}}(x) & =\sum_{a\in A_{m}\setminus\left\{ 0\right\} }\int_{B_{a}}f(x)\d{\mu}(x)+\int_{B_{0}}f\left(x\right)\d{\mu\left(x\right)}+\int_{E^{c}}f\left(x\right)\d{\mu}\left(x\right)\\
 & =\int_{\R^{n}}f\left(x\right)\d{\mu\left(x\right)}.
\end{align*}
Thus, $y':=\int_{\R^{n}}xg\left(x\right)\d{\mu}\left(x\right)\in\met{\mu_{m}}.$
A direct computation shows that 

\begin{align*}
\|y'-y\|_{B_{2}^{n}} & \le\sum_{a\in A_{m}}\int_{B_{a}}\norm{x-a}_{B_{2}^{n}}f(x)\d{\mu}(x)+\int_{E^{c}}\|x\|_{B_{2}^{n}}f\left(x\right)\d{\mu}\left(x\right)\\
 & \le\frac{\eps}{\mu\left(\R^{n}\right)}\sum_{a\in A_{m}}\mu\left(B_{a}\right)+\eps\\
 & \le2\eps.
\end{align*}
The reverse statement is also true. Namely, for any $y'\in\met{\mu_{m}}$
there exists $y\in\met{\mu}$ such that $\norm{y-y'}_{B_{2}^{n}}\le2\eps$.
Indeed, let $y'\in\met{\mu_{m}}$, and $0\le g\le1$ be a function
such that $\int_{\R^{n}}g\left(x\right)\d{\mu_{m}\left(x\right)}=1$
and $\int_{\R^{n}}g\left(x\right)\d{\mu_{m}\left(x\right)}=y'$. Correspondingly,
we define the function $f$ by setting

\[
f(x):=\begin{cases}
g\left(a\right), & x\in B_{a}\text{ for some }a\in A_{m}\\
g\left(0\right), & {\rm otherwise}
\end{cases}.
\]
Clearly, $0\le f\left(x\right)\le1$. Moreover, 
\begin{align*}
\int_{\R^{n}}f\left(x\right)\d{\mu\left(x\right)} & =\sum_{a\in A_{m}}\text{\ensuremath{\int_{B_{a}}g\left(a\right)\d{\mu\left(x\right)}+\int_{E^{c}}g\left(0\right)\d{\mu\left(x\right)}=\sum_{a\in A_{m}}\mu\left(B_{a}\right)g\left(a\right)+\mu\left(E^{c}\right)g\left(0\right)}}\\
 & =\sum_{a\in A_{m}\backslash\left\{ 0\right\} }\mu\left(B_{a}\right)g\left(a\right)+\left(\mu\left(B_{0}\right)+\mu\left(E^{c}\right)\right)g\left(0\right)\\
 & =\int_{\R^{n}}g\left(x\right)\d{\mu_{m}\left(x\right)}=1.
\end{align*}
Setting $y=\int_{\R^{n}}xf\left(x\right)\d{\mu_{m}\left(x\right)}$,
we obtain that
\begin{align*}
\norm{y-y'}{}_{B_{2}^{n}} & \le\sum_{a\in A_{m}}\int_{B_{a}}\norm{x-a}_{B_{2}^{n}}g(a)\d{\mu}(x)+\int_{E^{c}}g\left(0\right)\norm x_{B_{2}^{n}}\d{\mu\left(x\right)}\\
 & \le\frac{\eps}{\mu\left(\R^{n}\right)}\sum_{a\in A_{m}}\mu\left(B_{a}\right)+g\left(0\right)\eps\\
 & \le2\eps.
\end{align*}

Using the same argument as in the proof of Lemma \ref{lem:approx_infinite},
one can verify that 
\[
\left(1-2\eps\right)\met{\mu}\sub\met{\mu_{m}}\sub\left(1+2\eps\right)\met{\mu}.
\]
On the other hand, $\int_{\R^{n}}\norm x{}_{K}\d{\mu_{m}}(x)\le\int_{\R^{n}}\norm x{}_{K}\d{\mu}(x)+\eps$
is straightforward if one breaks down the integration to small partitions
$B_{a}$ and $E^{c}$. 

By replacing $\mu_{m}$ with the pushforward measure $\nu=\frac{I_{n}}{1-2\eps}\#\mu_{m}$,
it follows from Fact \ref{fact:met_lin_inv} that 
\[
\met{\mu}\sub\met{\nu}\sub\frac{\left(1+2\eps\right)}{1-2\eps}\met{\mu},
\]
and, in particular, $K\sub\met{\nu}$. Furthermore, 
\[
\int_{\R^{n}}\norm x{}_{K}\d{\nu}(x)=\frac{1}{1-2\eps}\int_{\R^{n}}\norm x{}_{K}\d{\mu_{m}}(x)\le\frac{1}{1-2\eps}\int_{\R^{n}}\norm x{}_{K}\d{\mu}(x)+\frac{\eps}{1-2\eps}.
\]
\end{proof}

\subsection{Scaling effect on discrete measure\textmd{s}}

Another property that we shall use in the sequel is the following
behavior of metronoids that are generated by discrete measures, under
scaling:
\begin{prop}
\label{prop:scaling_effect}Let $x_{1},\dots,x_{m}\in\R^{n}$, $a_{1},\dots,a_{m}\ge0$,
and $\mu=\sum_{i=0}^{m}a_{i}\delta_{x_{i}}$, where $x_{0}=0$, and
$a_{0}\ge0$ . Then for any choice of $r_{1},\dots,r_{m}\ge1$, the
measure $\nu=\sum_{i=1}^{m}\frac{a_{i}}{r_{i}}\delta_{r_{i}x_{i}}+\delta_{0}$
satisfies that $\met{\mu}\sub\met{\nu}$, where equality holds whenever
$\sum_{i=1}^{m}a_{i}\le1$ and $a_{0}\ge1$. Moreover, for any convex
body $K\sub\mathbb{R}^{n}$ containing $0$, we have that $\int_{\R^{n}}\norm x{}_{K}\d{\mu}(x)=\int_{\R^{n}}\norm x{}_{K}\d{\nu}(x)$. 
\end{prop}

\begin{proof}
Let $y\in\met{\mu}$. Then there exists a function $0\le f\le1$ such
that $\int_{\R^{n}}f(x)\d{\mu}(x)=1$, and $y=\int_{\R^{n}}xf(x)\d{\mu}(x)$.
We construct a function $g$, with support on $\left\{ 0,r_{1}x_{1},\cdots,r_{m}x_{m}\right\} $,
as follows; $g(r_{i}x_{i}):=f(x_{i})$ for $i\in\left\{ 1,\dots,m\right\} $,
and $g(0)=1-\sum_{i=1}^{m}\frac{a_{i}}{r_{i}}g(r_{i}x_{i})$. One
can easily verify that $0\le g\left(r_{i}x_{i}\right)\le1$ for all
$i\in\left\{ 1,\dots,m\right\} $, and that also $0\le g\left(0\right)\le1$,
due to the fact that $r_{i}\ge1$. Moreover, we have that 

\[
\int_{\R^{n}}g(x)\d{\nu}(x)=\sum_{i=1}^{m}\frac{a_{i}}{r_{i}}g(r_{i}x_{i})+\left(1-\sum_{i=1}^{m}\frac{a_{i}}{r_{i}}g(r_{i}x_{i})\right)=1,
\]
and

\[
\int_{\R^{n}}xg(x)\d{\nu}(x)=\sum_{i=1}^{m}r_{i}x_{i}g(r_{i}x_{i})\frac{a_{i}}{r_{i}}=\sum_{i=1}^{m}f(x_{i})a_{i}x_{i}=y,
\]
as claimed. 

Next, assume that $\sum_{i=1}^{m}a_{i}\le1$ and $a_{0}\ge1$. Let
$z\in\met{\nu}$. Then there exists a function $0\le f\le1$ such
that $\sum_{i=0}^{m}\frac{a_{i}}{r_{i}}f(r_{i}x_{i})=1$ and $\sum_{i=0}^{m}a_{i}f(r_{i}x_{i})=z$.
We define a new function $0\le g\le1$ whose support is $\left\{ {0,x_{1},\cdots,x_{n}}\right\} $
by setting $g(x_{i}):=f(r_{i}x_{i})$ for $i\in\left\{ 1,\dots,m\right\} $,
and $g\left(0\right)=1-\sum_{i=1}^{m}a_{i}g(x_{i})$. Thus, we have
that $0\le g\le1$, $\sum_{i=0}^{m}a_{i}g(x_{i})=1$, and $\sum_{i=0}^{m}a_{i}g(x_{i})x_{i}=z$.
Therefore, $\met{\nu}\sub\met{\mu}$, and hence $\met{\nu}=\met{\mu}$.\\

Finally, let $K\sub\R^{n}$ be a convex body containing $0$. Since
for any $r\ge0$ and $x\in\R^{n}$, we have that $\norm{rx}{}_{K}=r\norm x{}_{K}$,
it follows that 

\[
\int_{\R^{n}}\norm x{}_{K}\d{\mu}(x)=\sum_{i=1}^{m}a_{i}\norm{x_{i}}{}_{K}=\sum_{i=1}^{m}\frac{a_{i}}{r_{i}}\norm{r_{i}x_{i}}{}_{K}=\int_{\R^{n}}\norm x{}_{K}\d{\nu}(x).
\]
 
\end{proof}
The following observation is an immediate consequence of Proposition
\ref{prop:scaling_effect}:
\begin{cor}
Let $K\sub\R^{n}$ be a convex body containing $0$. Suppose $\mu$
is a finite discrete measure such that $K\sub\met{\mu}$. Then there
exists a discrete probability measure $\nu$ such that $\text{K\ensuremath{\sub}}\met{\nu+\delta_{0}}$,
and $\int_{\R^{n}}\norm x{}_{K}\d{\mu}(x)=\int_{\R^{n}}\norm x{}_{K}\d{\nu}(x)$.
\end{cor}

\begin{proof}
Suppose $\mu=\sum_{i=0}^{m}a_{i}\delta_{x_{i}}$ with $a_{i}\ge0$
and $x_{0}=0$, and let $r=\mu\left(\R^{n}\right)=\sum_{i=0}^{m}a_{i}\ge1$.
Then Proposition \ref{prop:scaling_effect} implies that the measure
$\nu=\sum\frac{a_{i}}{r}\delta_{rx_{i}}$ satisfies our claim.
\end{proof}

\section{Estimating $d_{R}^{*}\left(K\right)$ and $D_{R}^{*}\left(K\right)$
\label{sec:Estimating_d_D}}

In this section we prove Theorems \ref{thm:main_upper_sym} and \ref{thm:main_upper_bdd}. 

\subsection{Proof of Theorem \ref{thm:main_upper_sym}}

Let $\sigma_{r}$ denote the uniform probability measure on $r\Sph^{n-1}$.
For $r=1$ we simply denote $\sigma=\sigma_{1}$. 
\begin{proof}
[Proof of Theorem \ref{thm:main_upper_sym}] Let $K\sub\R^{n}$ be
a centrally-symmetric convex body such that $B_{2}^{n}$ is the minimal
volume circumscribed ellipsoid of $K$. By John's theorem (see e.g.,
\cite{AGM15}), we have that $\frac{1}{\sqrt{n}}B_{2}^{n}\sub K\sub B_{2}^{n}$. 

Consider the measure $\mu=2\sigma_{R}$, where $R=\left(\int_{\Sph^{n-1}}\left|\iprod x{\theta}\right|\d{\sigma}\left(x\right)\right)^{-1}$.
Let $e_{1}\in S^{n-1}$, and define $f\left(x\right)=\one_{\left\{ y\,:\,\iprod y{e_{1}}>0\right\} }\left(x\right)$.
Then we have $\int_{\R^{n}}f\left(x\right)\d{\mu}\left(x\right)=1$,
and therefore,
\[
h_{M\left(\mu\right)}\left(e_{1}\right)\ge\int_{\iprod x{e_{1}}>0}\iprod x{e_{1}}\d{\mu}(x)=\int_{\R^{n}}\left|\langle x,\,e_{1}\rangle\right|\d{\sigma_{R}}\left(x\right)=R\int_{\Sph^{n-1}}\left|\iprod x{\theta}\right|\d{\sigma}\left(x\right)=1,
\]
which implies that $B_{2}^{n}\sub\met{\mu}$. In fact, Proposition
\ref{ExtremePointofMu} tells us that $h_{M\left(\mu\right)}\left(e_{1}\right)=1$,
which means that $\met{\mu}=B_{2}^{n}$, and hence $\frac{1}{\sqrt{n}}\met{\mu}\sub K\sub\met{\mu}$. 

Finally, note that $\mu\left(\R^{n}\right)=2$. Moreover, by a standard
computation, one can verify that 

\begin{equation}
\int_{\Sph^{n-1}}\left|\iprod{\theta}{e_{1}}\right|\d{\sigma}\left(\theta\right)=\sqrt{\frac{2}{\pi n}}\left(1+{\rm o}\left(1\right)\right).\label{eq:aver_inner_prod}
\end{equation}
Therefore, 
\[
\int_{\R^{n}}\norm x_{K}\d{\mu}\left(x\right)\le\sqrt{n}\int_{\R^{n}}\norm x_{B_{2}^{n}}\d{\mu}\left(x\right)\le Cn,
\]
which completes our proof. 
\end{proof}

\subsection{Proof of Theorem \ref{thm:main_upper_bdd}}

To prove Theorem \ref{thm:main_upper_bdd}, we need the following
consequence of the Brunn-Minkowski theorem which was observed (in
equivalent forms) several times in the literature, e.g., in \cite{Mit69},
and \cite{Szarek2014}. For the sake of completeness, we provide a
proof.
\begin{prop}
\label{TailBound_Volume} Let $K\sub\mathbb{R}^{n}$ be a centered
convex body. Fix $R>1$, $u\in\Sph^{n-1}$, and $r=h_{K}\left(u\right)$.
Let $L:=K\cap\left\{ x\in\R^{n}\,:\,\iprod xu\ge\frac{r}{R}\right\} $.
Then
\[
\frac{\text{vol}(L)}{\text{vol}(K)}\ge\exp\left(-1-\frac{n-1}{R-1}\right).
\]
\end{prop}

To prove Proposition \ref{TailBound_Volume}, we need the following
lemma. Given $R>1$, and a non-negative concave function $f:\left[0,R\right]\rightarrow\R$,
let $\widetilde{f}:\left[0,R\right]\to\R$ denote the linear function
satisfying $\tilde{f}\left(1\right)=f\left(1\right)$ and $\tilde{f}\left(R\right)=0$. 
\begin{lem}
\label{TailBound_Convexity} For $R>1$, let $f:\left[0,R\right]\rightarrow\mathbb{\R}$
be a non-negative concave function. Then for any increasing function
$g:\left[0,\infty\right)\rightarrow\left[0,\infty\right)$, we have
\[
\frac{\int_{1}^{R}g\circ f(t)\d t}{\int_{0}^{R}g\circ f(t)\d t}\ge\frac{\int_{1}^{R}g\circ\tilde{f}(t)\d t}{\int_{0}^{R}g\circ\tilde{f}(t)\d t}.
\]
\end{lem}

\begin{proof}[Proof of Lemma \ref{TailBound_Convexity} ]
Let $A_{1}:=\int_{0}^{1}g\circ f(t)\d t$ and $A_{2}:=\int_{1}^{R}g\circ f(t)\d t$.
In particular, 
\[
\frac{\int_{1}^{R}g\circ f(t)\d t}{\int_{0}^{R}g\circ f(t)\d t}=\frac{A_{2}}{A_{1}+A_{2}}.
\]
Our next goal is to bound $A_{1}$ from above and $A_{2}$ from below.
To bound $A_{1}$ from above, note that since $f$ is concave and
non-negative, we have that for any $t\in[0,1]$,
\begin{eqnarray*}
f(t) & \le & f(1)-(1-t)\frac{f(R)-f(1)}{R-1}\\
 & \le & f(1)+(1-t)\frac{f(1)}{R-1}=\tilde{f}(t).
\end{eqnarray*}
Since $g$ is increasing, we thus obtain 
\[
A_{1}\le\int_{0}^{1}g(\tilde{f}(t))\d t.
\]
Similarly, we bound $A_{2}$ from above by noting that for any $t\in[1,R]$,
\begin{eqnarray*}
f(t) & \ge & f(1)+(t-1)\frac{f(R)-f(1)}{R-1}\\
 & \ge & f(1)-(t-1)\frac{f(1)}{R-1}=\tilde{f}(t),
\end{eqnarray*}
and hence
\[
A_{2}\ge\int_{1}^{R}g(\tilde{f}(t))\d t.
\]
Finally, the above bounds for $A_{1}$ and $A_{2}$ imply that 
\[
\frac{A_{2}}{A_{1}+A_{2}}\ge\frac{\int_{1}^{R}g(\tilde{f}(t))\d t}{\int_{0}^{1}g(\tilde{f}(t))\d t+\int_{1}^{R}g(\tilde{f}(t))\d t}.
\]
\end{proof}
\begin{proof}[Proof of Proposition \ref{TailBound_Volume} ]
We may rescale $K$ so that $r=R$. Let 
\[
K_{t}:=\left\{ x\in\mathbb{R}^{n}\,,\,\iprod xu=t\right\} ,\text{ }K_{+}:=K\cap\left\{ x\in\mathbb{R}^{n},\iprod xu\ge0\right\} ,
\]
and $f(t):=\text{vol}(K_{t})^{\frac{1}{n-1}}$. By Brunn-Minkowski
theorem, $f(t)$ is a concave function on its support. Moreover, we
clearly have that $\text{vol}(K_{+})=\int_{0}^{R}f^{n-1}(t)\d t$
and $\text{vol}(L)=\int_{1}^{R}f^{n-1}(t)\d t$. Also note that $\tilde{f}(t)=\frac{f(1)}{1-\frac{1}{R}}(1-\frac{t}{R})$.
Therefore, Lemma \ref{TailBound_Convexity}, applied with $f\left(t\right)$
and $g\left(t\right)=t^{n-1}$, implies that
\begin{eqnarray*}
\frac{\text{vol}(L)}{\text{vol}(K_{+})} & \ge & \frac{\int_{1}^{R}(1-\frac{t}{R})^{n-1}\d t}{\int_{0}^{R}(1-\frac{t}{R})^{n-1}\d t}=\frac{\frac{R}{n}(1-\frac{1}{R})^{n-1}}{\frac{R}{n}}\\
 & = & (1-\frac{1}{R})^{n-1}=(\frac{R}{R-1})^{-(n-1)}\\
 & = & \left((1+\frac{1}{R-1})^{n-1}\right)^{-1}\\
 & \ge & \exp(-\frac{n-1}{R-1}),
\end{eqnarray*}
where the last inequality relies on the fact that $1+x\le e^{x}$.
Since, by Grünbaum \cite{Gr60}, we know that $\frac{\text{vol}(K_{+})}{\text{vol}(K)}\ge\frac{1}{e}$,
our proof is complete.
\end{proof}
We are now ready to prove Theorem \ref{thm:main_upper_bdd}:
\begin{proof}[Proof of Theorem \ref{thm:main_upper_bdd} ]
 Let $K\sub\mathbb{R}^{n}$ be a centered convex body, and let $\mu$
be the uniform measure on $RK$, satisfying $\d{\mu}\left(x\right)=\frac{\exp\left(1+\frac{n-1}{R-1}\right)}{\vol{RK}}\one_{RK}(x)\d x$.
Define 
\[
L_{\theta}:=RK\cap\left\{ x\in\mathbb{R}^{n}\,,\,\iprod x{\theta}\ge h_{K}(\theta)\right\} .
\]
 By Proposition \ref{TailBound_Volume}, we have that 
\begin{equation}
\vol{L_{\theta}}\ge\exp\left(-1-\frac{n-1}{R-1}\right)\vol{RK}.\label{eq:vol_Ltheta}
\end{equation}
In particular, we get that $\mu\left(L_{\theta}\right)=\frac{\exp\left(1+\frac{n-1}{R-1}\right)}{\vol{RK}}\vol{L_{\theta}}\ge1$.

Fix $\theta\in\Sph^{n-1}$, and let $f_{\theta}\left(x\right)=\frac{1}{\mu(L_{\theta})}\mathbf{1}_{L_{\theta}}.$
By the previous argument, $0\le f_{\theta}\le1$, and $\int_{\R^{n}}f_{\theta}(x)\d{\mu}(x)=1$.
Therefore, $x_{\theta}:=\int_{\mathbb{R}^{n}}xf_{\theta}\left(x\right)\d{\mu}(x)\in\met{\mu}$.
In particular, by the definition of $L_{\theta}$, it follows that
\[
h_{\met{\mu}}\left(\theta\right)\ge\iprod{x_{\theta}}{\theta}\ge h_{K}\left(\theta\right)\int_{\R^{n}}f_{\theta}\left(x\right)\d{\mu}\left(x\right)=h_{K}\left(\theta\right),
\]
and therefore $K\sub\met{\mu}$. 

Finally, we have that $d_{R}^{*}\left(K\right)\le\mu\left(\R^{n}\right)=\exp\left(1+\frac{n-1}{R-1}\right)$,
and 
\[
D_{R}^{*}\left(K\right)\le\int_{\mathbb{R}^{n}}\norm x{}_{K}\d{\mu}(x)\le R\mu\left(\R^{n}\right)=R\exp\left(1+\frac{n-1}{R-1}\right).
\]
Moreover, it follows that $\fvein K\le D_{n}^{*}\left(K\right)\le e^{2}n$.
\end{proof}

\section{the fractional vertex index \label{sec:fvein}}

\subsection{\label{sec:fvein_exact_comp}A couple of extremal examples}

\subsubsection{\textbf{The vertex index of the cross-polytope}}
\begin{prop}
\label{prop:fvein_B1}We have that $\fvein{B_{1}^{n}}=2n$. 
\end{prop}

\begin{proof}
We follow the lines of the proof in \cite{BezLitvak07} that $\vein{B_{1}^{n}}=2n$.
Let $\norm{\cdot}_{1}$ denote the norm induced by $B_{1}^{n}$, that
is $\norm x_{1}=\sum_{i=1}^{n}\left|x_{i}\right|=\sum_{i=1}^{n}\left|\iprod x{e_{i}}\right|$,
where $e_{1},\dots,e_{n}$ are the standard basis of $\R^{n}$. Let
$\mu$ be a measure such that $B_{1}^{n}\sub\met{\mu}$. Then for
each $e_{i}$ there exists a function $0\le f\le1$ such that $e_{i}=\int_{\R^{n}}xf(x)\d{\mu}(x)$,
and hence $1=\int_{\R^{n}}\iprod x{e_{i}}f(x)\d{\mu}(x)$. Furthermore,
if we define the function $g$ by 

\[
g(x)=\begin{cases}
f(x), & \text{if }\iprod x{e_{i}}\ge0\\
0, & \text{otherwise}
\end{cases},
\]
we obtain the following inequality: $1\le\int_{\R^{n}}\langle x,\,e_{i}\rangle g(x)\d{\mu}(x)\le\int_{\R^{n}}\max\left\{ \langle x,\,e_{i}\rangle,0\right\} \d{\mu}(x)$.
Applying the same argument to $-e_{i}$, we have $1\le\int_{\R^{n}}\max\left\{ \langle x,\,-e_{i}\rangle,0\right\} \d{\mu}(x)$.
Therefore, it follows that 

\[
\int_{\R^{n}}\norm x_{1}\d{\mu}(x)=\sum_{i=1}^{n}\int_{\R^{n}}\left|\iprod x{e_{i}}\right|\d{\mu}(x)\ge2n.
\]

On the other hand, if $\mu=\sum_{i=1}^{n}(\delta_{e_{i}}+\delta_{-e_{i}})$,
then $B_{1}^{n}=\met{\mu}$, and $\int_{\R^{n}}\norm x{}_{1}\d{\mu}(x)=2n$.
Therefore, the lower bound is attained by $\mu$.
\end{proof}

\subsubsection{\textbf{The vertex index of the Euclidean ball}}
\begin{prop}
\label{prop:vein_Ball}We have that $\fvein{B_{2}^{n}}=\sqrt{2\pi n}\left(1+{\rm o}\left(1\right)\right).$
\end{prop}

\begin{proof}
The upper bound $\fvein{B_{2}^{n}}\le\sqrt{2\pi n}\left(1+{\rm o}\left(1\right)\right)$
follows verbatim from the proof of theorem \ref{thm:main_upper_sym}
by considering the measure $\mu=2\sigma_{R}$, where $R=\left(\int_{\Sph^{n-1}}\left|\iprod x{\theta}\right|\d{\sigma}\left(x\right)\right)^{-1}$,
which implies that $\met{\mu}=B_{2}^{n}$ , and 
\[
\int_{\mathbb{R}^{n}}\norm x{}_{B_{2}^{n}}\d{\mu}=2R=\sqrt{2\pi n}\left(1+o\left(1\right)\right).
\]

Next, we show that $\fvein{B_{2}^{n}}\ge\sqrt{2\pi n}\left(1+{\rm o}\left(1\right)\right)$.
Let $\mu$ be any measure satisfying that $K\sub\met{\mu}$. By Lemmas
\ref{lem:approx_infinite}, \ref{lem:appx_discrete}, and Proposition
\ref{prop:scaling_effect}, we may assume without loss of generality
that $\mu$ is discrete, finite, and that $\supp{\mu}\sub r\Sph^{n-1}\cup\left\{ 0\right\} $
for some $r>0$. By adding $\delta_{0}$ to $\mu$ at no additional
cost, we may also assume that $\mu$ has an atom $D_{0}\delta_{0}$
at the origin.

Let ${\rm SO}\left(n\right)$ be the rotation group on $\R^{n}$,
and let $\xi$ be the normalized probability Haar measure on ${\rm SO}\left(n\right)$.
We define the radial measure $\mu_{0}$ by letting
\[
\mu_{0}\left(A\right)=\int_{{\rm SO}\left(n\right)}\mu\left(u^{-1}A\right)\d{\xi}\left(u\right)
\]
for any Borel set $A\sub\R^{n}$. Note that $\mu_{0}=D_{0}\delta_{0}+D_{r}\sigma_{r}$,
where $D_{r}r=\int_{\R^{n}}\norm x_{B_{2}^{n}}\d{\mu}\left(x\right)$.
By Proposition \ref{prop:necessary_1}, for any $\theta\in\Sph^{n-1}$
we have 
\begin{align*}
\int_{H_{\theta}^{+}}\iprod x{\theta}\d{\mu_{0}}\left(x\right) & =\int_{{\rm SO}\left(n\right)}\int_{u^{-1}H_{\theta}^{+}}\iprod{ux}{\theta}\d{\mu}\left(x\right)\d{\xi}\left(u\right)=\int_{{\rm SO}\left(n\right)}\int_{H_{u^{-1}\theta}^{+}}\iprod x{u^{-1}\theta}\d{\mu}\left(x\right)\d{\xi}\left(u\right)\\
 & \ge\int_{{\rm SO}\left(n\right)}\d{\xi}\left(u\right)=1.
\end{align*}
Combined with \eqref{eq:aver_inner_prod}, the above inequality implies
that 
\[
2\le\int_{\R^{n}}\left|\iprod x{e_{1}}\right|\d{\mu_{0}}\left(x\right)=D_{r}r\int_{\Sph^{n-1}}\left|\iprod{\theta}{e_{1}}\right|\d{\sigma}\left(\theta\right)=D_{r}r\sqrt{\frac{2}{\pi n}}\left(1+{\rm o}\left(1\right)\right).
\]
Therefore, it follows that $\int_{\R^{n}}\norm x_{B_{2}^{n}}\d{\mu}\left(x\right)=D_{r}r\ge\sqrt{2\pi n}\left(1+{\rm o}\left(1\right)\right)$,
as claimed.  
\end{proof}
\begin{rem}
A shorter argument provides the slightly worst lower bound $\fvein{B_{2}^{n}}\ge2\sqrt{n}$.
Indeed, since $d\left(B_{2}^{n},\,B_{1}^{n}\right)=\sqrt{n}$, Proposition
\ref{prop:fvein_B1} and Fact \ref{fact:BM_distance} imply that 
\[
\fvein{B_{2}^{n}}\ge\frac{\fvein{B_{1}^{n}}}{\sqrt{n}}=2\sqrt{n}.
\]
\end{rem}

\subsection{A Lower bound }

This section is devoted for the proof of Theorem \ref{thm:fvein_lower_bdd}. 

We will need the following fact which relates the fractional vertex
index of two convex bodies through their Banach-Mazur distance. Let
$I_{n}:\R^{n}\to\R^{n}$ denote the identity operator on $\R^{n}$.
Let $d\left(K,L\right)$ denote the Banach-Mazur distance between
two centrally-symmetric convex bodies, $K,L\sub\R^{n}$. In \cite{BezLitvak07},
the authors show that $\vein K\le\vein Ld\left(K,L\right)$. Analogously,
we have:
\begin{fact}
\label{fact:BM_distance}Let $K,L$ be centrally-symmetric convex
bodies in $\R^{n}$. Then 
\[
\fvein K\le d\left(K,L\right)\fvein L.
\]
\end{fact}

\begin{proof}
Let $T$ be some invertible linear transformation such that $TL\sub K\sub d\left(K,L\right)TL$.
Suppose $\mu$ is a measure satisfying that $TL\sub\met{\mu}$. Then,
by Fact \ref{fact:met_lin_inv}, $K\sub\met{d(K,L)\cdot I_{n}\#\mu}$
and hence 
\[
\fvein K\le d\left(K,L\right)\int\norm x_{\R^{n}K}\d{\mu}\left(x\right)\le\int_{\R^{n}}\norm x_{L}\d{\mu}\left(x\right).
\]
Since $\fvein L$ is linear-invariant, and $\mu$ is arbitrary, it
follows that 
\[
\fvein K\le d\left(K,L\right)\fvein{TL}=d\left(K,L\right)\fvein L.
\]
\end{proof}
We shall also use the following proportional Dvoretzky-Rogers factorization
Theorem by Bourgain and Szarek:
\begin{thm*}
[Bourgain and Szarek \cite{BourSz88}] If $\left(X,\|\cdot\|\right)$
is an $n$-dimensional normed space and $\epsilon\in(0,1)$, there
exists vectors $x_{1},\dots,x_{m}\in X$, $m\ge(1-\epsilon)n$, such
that for any real $t_{1},\dots,t_{n}$, 
\[
\max_{j\le m}|t_{j}|\le\norm{\sum_{j\le m}t_{j}x_{j}}_{X}\le c\epsilon^{-d}\left(\sum_{j\le m}t_{j}^{2}\right)^{\frac{1}{2}},
\]
where $c,d>0$ are absolute constants. 
\end{thm*}
Let us fix an orthogonal basis $\left\{ e_{1},\dots,e_{n}\right\} $
of $\R^{n}$, and $\epsilon=\frac{1}{2}$. Given a subspace $E\sub\R^{n}$
and $1\le p\le\infty$, denote $B_{p}^{E}=B_{p}^{n}\cap E$. For our
purpose, it will be enough to use the following simpler geometric
version of the above proportional Dvoretzky-Rogers factorization Theorem:
\begin{thm}
\label{thm:DvoretzkyRogerGeometricVersion} Let $K\subset\R^{n}$
be a centrally-symmetric convex body. Let $E\sub\R^{n}$ be the subspace
spanned by $e_{1},\cdots,e_{\left\lceil n/2\right\rceil }$. Then
there exists a linear transformation $T\in{\rm GL}_{n}\left(\R\right)$
such that 
\[
cB_{2}^{E}\sub TK\cap E\sub B_{\infty}^{E},
\]
where $c>0$ is a universal constant. 
\end{thm}

We are now ready to prove Theorem \ref{thm:fvein_lower_bdd}:
\begin{proof}[Proof of Theorem \ref{thm:fvein_lower_bdd} ]
 By Theorem \ref{thm:DvoretzkyRogerGeometricVersion}, applied to
$K^{\circ}$, and the fact that $B_{\infty}^{n}\sub\sqrt{n}B_{2}^{n}$,
there exists $T\in{\rm GL}{}_{n}\left(\R\right)$ such that 
\begin{equation}
\frac{c}{\sqrt{n}}B_{\infty}^{E}\sub TK^{\circ}\cap E\sub B_{\infty}^{E}.\label{eq:DvoFact_Inter}
\end{equation}

Let $\proj_{E}:\R^{n}\to E$ denote the orthogonal projection onto
$E$, and set $\widetilde{K}=\left(T^{-1}\right)^{\intercal}K$. By
the properties of polarity, \eqref{eq:DvoFact_Inter} is equivalent
to
\begin{align*}
B_{1}^{E} & \sub\proj_{E}\widetilde{K}\sub\frac{\sqrt{n}}{c}B_{1}^{E}.
\end{align*}
In other words, we have that $d\left(\proj_{E}K,\,B_{1}^{E}\right)\le\frac{\sqrt{n}}{c}$
.

Next, fix $\eps>0$. By Lemmas \ref{lem:approx_infinite} and \ref{lem:appx_discrete},
there exists a finite discrete measure $\mu$ of the form $\mu=\sum_{i=1}^{N}w_{i}\delta_{x_{i}}$
such that $K\sub\met{\mu}$ and $\int_{\R^{n}}\norm x_{K}\d{\mu}\left(x\right)\le\fvein K+\eps$.
Consider the measure $\nu$ on $E$ defined by $\nu=\sum_{i=1}^{N}w_{i}\delta_{\proj_{E}\left(x_{i}\right)}$.
One can verify that $\proj_{E}\left(K\right)\sub\met{\nu}$. Moreover,
since $\norm x_{K}\ge\norm{\proj_{E}\left(x\right)}_{\proj_{E}\left(K\right)}$
for any $x\in\R^{n}$, it follows that 
\[
\fvein K+\eps\ge\int_{\R^{n}}\norm x_{K}\d{\mu}\left(x\right)\ge\int_{E}\norm y_{\proj_{E}\left(K\right)}\d{\nu}\left(y\right)\ge\fvein{\proj_{E}\left(K\right)}.
\]
On the other hand, by Fact \ref{fact:BM_distance}, we have that 
\[
\fvein{P_{E}K}\ge\frac{\fvein{B_{1}^{E}}}{d\left(B_{1}^{E},\,\proj_{E}K\right)}\ge c\sqrt{n},
\]
where we used $\fvein{B_{1}^{E}}=2\dim\left(E\right)=n$ in the last
inequality. 
\end{proof}
\begin{rem}
Note that the upper bound $\fvein K\le e^{2}n$, for any convex body
$K$, immediately follows from Corollary \ref{cor:upper_bdd}.
\end{rem}

\section{An application to centroid bodies \label{sec:app_centroid}}

In this section we show how Corollary \ref{cor:centrod_low_bdd} is
a direct consequence of Theorem \ref{thm:fvein_lower_bdd}. 

\subsection{Reformulating $\protect\fvein K$}

Let $\mu$ be a Borel measure such that $K\sub\met{\mu}$. By Lemmas
\ref{lem:approx_infinite}, \ref{lem:appx_discrete}, and Proposition
\ref{prop:scaling_effect}, we may assume without loss of generality
that $\mu$ is a finite discrete measure with $\mu\left(\R^{n}\right)=2$
and $\mu\left(\left\{ 0\right\} \right)=1$.  Let $\mathscr{D}_{n}$
be the collection of all non-degenerate finite discrete probability
measure. Then, the previous argument implies that 

\[
\text{vein}^{*}(K)=\inf\left\{ \int_{\R^{n}}\norm x_{K}\d{\mu}\left(x\right)\,:\,\mu\in\mathscr{D}_{n},\,K\sub\met{\mu+\delta_{0}}\right\} .
\]
Moreover, if $K=-K$ then for each measure $\mu$ such that $K\sub\met{\mu}$,
we can define the symmetric measure $\nu$ by $\nu(A)=\frac{\mu(A)+\mu(-A)}{2}$
for any Borel set $A\sub\R^{n}$. Then $\nu$ is a symmetric measure,
satisfying that $K\subset\met{\nu}$ and $\int_{\R^{n}}\norm x{}_{K}\d{\nu}(x)=\int_{\R^{n}}\norm x{}_{K}\d{\mu}(x).$
Therefore, we conclude that 

\begin{equation}
\text{vein}^{*}(K)=\inf\left\{ \int_{\R^{n}}\norm x_{K}\d{\mu}\left(x\right)\,:\,\mu\in\mathscr{D}_{n}^{s},\,K\sub\met{\mu+\delta_{0}}\right\} ,\label{eq:fvein__prob}
\end{equation}
where $\mathscr{D}_{n}^{s}$ is the collection of all symmetric non-degenerate
finite discrete probability measures. 

In view of Remark \ref{rem:fvein_zonotope}, the above equality immediately
implies the following reformulation of the fractional vertex index:
\begin{prop}
For any convex body $K\sub\R^{n}$, we have

\[
\fvein K=\inf\left\{ \sum_{i=1}^{m}\norm{y_{i}}{}_{K}\:,\:K\sub Z(y_{1},\cdots,y_{m})\right\} .
\]
\end{prop}

\subsection{A relation to $L_{1}$-centroid bodies }

Let $\mathcal{K}_{n}$ be the class of all symmetric convex bodies
in $\R^{n}$, and $\mathscr{F}_{n}$ the class of all non-degenerate
Borel probability measures on $\R^{n}$ with bounded first moment.
We have the following equivalence:
\begin{prop}
\label{prop:vein_Z1_equiv}~${\displaystyle \inf_{K\in\mathcal{K}_{n}}\fvein K=2\inf_{\mu\in\mathscr{F}_{n}}\int_{\R^{n}}\norm x{}_{Z_{1}(\mu)}\d{\mu}\left(x\right).}$
\end{prop}

\begin{proof}
By \eqref{eq:fvein__prob}, we have that
\[
\fvein K=\inf\left\{ \int_{\R^{n}}\norm x{}_{K}\d{\mu}\left(x\right)\,:\,\mu\in\mathscr{D}_{n}^{s},K\sub\met{\mu+\delta_{0}}\right\} .
\]
Moreover, Proposition \ref{prop:MuCentroid-1} implies that for any
$\mu\in\mathscr{D}_{n}^{s}$, $\met{\mu+\delta_{0}}=\frac{1}{2}Z_{1}(\mu)$.
Therefore,
\begin{align*}
\inf_{K\in\mathcal{K}_{n}}\fvein K & =\inf_{K\in\mathcal{K}_{n}}\inf\left\{ \int_{\R^{n}}\norm x{}_{K}\d{\mu}\left(x\right)\,:\:\mu\in\mathscr{D}_{n}^{s},\:K\sub\frac{1}{2}Z_{1}(\mu)\right\} \\
 & \ge\inf_{K\in\mathcal{K}_{n}}\inf\left\{ \int_{\R^{n}}\norm x_{\frac{1}{2}Z_{1}\left(\mu\right)}\d{\mu}\left(x\right)\,:\:\mu\in\mathscr{F}_{n},\:K\sub\frac{1}{2}Z_{1}(\mu)\right\} .
\end{align*}
By observing that 
\[
\inf\left\{ \int_{\R^{n}}\norm x_{\frac{1}{2}Z_{1}\left(\mu\right)}\d{\mu}\left(x\right)\,:\:\mu\in\mathscr{F}_{n},\:K\sub\frac{1}{2}Z_{1}(\mu)\right\} \ge\inf\left\{ \int_{\R^{n}}\norm x_{\frac{1}{2}Z_{1}\left(\mu\right)}\d{\mu}\left(x\right)\,:\:\mu\in\mathscr{F}_{n}\right\} ,
\]
which does not depend on $K$, we obtain that 
\begin{align*}
\inf_{K\in\mathcal{K}_{n}}\fvein K & \ge\inf\left\{ \int_{\R^{n}}\norm x_{\frac{1}{2}Z_{1}\left(\mu\right)}\d{\mu}\left(x\right)\,:\:\mu\in\mathscr{F}_{n}\right\} \\
 & \ge\inf_{\mu\in\mathscr{F}_{n}}\fvein{\frac{1}{2}Z_{1}(\mu)}\\
 & \ge\inf_{K\in\mathcal{K}_{n}}\fvein K
\end{align*}
To conclude, we have that 
\[
\inf_{K\in\mathcal{K}_{n}}\fvein K=\inf\left\{ \int_{\R^{n}}\norm x_{\frac{1}{2}Z_{1}\left(\mu\right)}\d{\mu}\left(x\right)\,:\:\mu\in\mathscr{F}_{n}\right\} =2\inf_{\mu\in\mathscr{F}_{n}}\int_{\R^{n}}\norm x{}_{Z_{1}(\mu)}\d{\mu}\left(x\right),
\]
as claimed.
\end{proof}
Finally, note that Corollary \ref{cor:centrod_low_bdd} follows directly
from Theorem \ref{thm:fvein_lower_bdd} and Proposition \ref{prop:vein_Z1_equiv}.

\noindent \bibliographystyle{plain}
\bibliography{MATH_June_05_2017}

\end{document}